\documentclass[a4paper,11pt]{amsart}
\usepackage[left=2.5cm, right=2.5cm,top=2.5cm,bottom=3cm]{geometry}
\usepackage[utf8]{inputenc}
\usepackage[english]{babel}
\usepackage{amssymb}
\usepackage{mathtools}
\usepackage{amsfonts}
\usepackage{amsmath}
\usepackage{amsthm}
\usepackage{dsfont}
\usepackage{soul}

\usepackage[doi=false, url=false, isbn=false,giveninits=true,date=year,style=ieee,clearlang=true,maxbibnames=99,sorting=nyt,backend=biber]{biblatex}
\usepackage{hyperref}
\usepackage{tikz}
\usepackage{tikz-cd}
\usetikzlibrary{cd}

\newtheorem{theoremintro}{Theorem}

\newtheorem{theorem}{Theorem}[section] 
 
\newtheorem{proposition}[theorem]{Proposition} 
\newtheorem{lemma}[theorem]{Lemma}

\theoremstyle{definition}
\newtheorem{definition}[theorem]{Definition}
\newtheorem{example}[theorem]{Example}
\newtheorem{remark}[theorem]{Remark}

\DeclareMathOperator{\initial}{in}
\DeclareMathOperator{\interior}{int}
\DeclareMathOperator{\conv}{conv}
\DeclareMathOperator{\supp}{supp}
\DeclareMathOperator{\diag}{diag}
\DeclareMathOperator{\relint}{relint}
\DeclareMathOperator{\Newt}{Newt}
\DeclareMathOperator{\Log}{Log}
\DeclareMathOperator{\Exp}{Exp}

\newcommand{\ZZ}{\mathbb{Z}}
\newcommand{\NN}{\mathbb{N}}
\newcommand{\QQ}{\mathbb{Q}}
\newcommand{\RR}{\mathbb{R}}
\newcommand{\CC}{\mathbb{C}}
\newcommand{\PP}{\mathbb{P}}

\newcommand{\cC}{\mathcal{C}}

\newcommand{\cox}{\mathrm{cox}}

\usepackage{comment}

\title{Toric extensions of Pólya's theorem}
\author{Lorenzo Baldi, Rainer Sinn, Máté L. Telek, Julian Weigert}
\date{November 2025}
\address{Universit\"at Leipzig and Max Plank Institute for Mathematics in the Sciences, Leipzig, Germany}

\stdpunctuation

\begin{document}

\begin{abstract}
The classical version of Pólya's theorem provides a simple method for certifying that a homogeneous polynomial of degree $d$ is strictly copositive, that is, it takes only positive values on the nonnegative real orthant. However, this method might fail to detect copositivity of polynomials that are missing certain degree $d$ monomials. In this paper, we present extensions and converses to Pólya's theorem for sparse polynomials, using techniques from positive toric geometry. Furthermore, we explore how this method can be used to study the convergence of Feynman integrals in particle physics.
\end{abstract}

\maketitle

\section{Introduction}
Certifying the nonnegativity of polynomials has a rich history in real algebraic geometry, dating back at least to Hilbert's 17th problem \cite{Hilbert}. Hilbert asked whether every globally nonnegative polynomial can be represented as a sum of squares of rational functions --- a question that was answered affirmatively by Artin \cite{Artin}. Around the same time, a similar representation theorem was established by Pólya \cite{Polya}. Pólya proved that for every \emph{strictly copositive} form $f \in \RR[t_1,\dots,t_n]$ --- that is, every homogeneous polynomial $f$ that takes only positive values on $\RR^n_{\geq0}\setminus \{0\}$ --- there exists an integer $N \in \NN$ such that $(t_1+\dots+t_n)^Nf$ has positive coefficients.

The term \emph{copositive} traces back to the work of Motzkin \cite{Motzkin}. Since Pólya and Motzkin, copositive polynomials have remained an active area of research in optimization \cite{BomzeDur,Dur}. It is known that deciding membership in the cone $\cC_{n,d}$ of copositive polynomials in $n$ variables of degree $d$ is NP-hard in general, and that $\cC_{n,2}$ is not a spectrahedral shadow for $n \geq 5$ \cite{BodirskyKummerThom}. For a detailed overview on copositivity, we refer the reader to \cite{Vargas,Parrilo} and the references therein.

For a strictly copositive $f$, lower bounds on the smallest exponent $N$ such that $(t_1+ \dots + t_n)^N f$ has positive coefficients were provided in \cite{PowersReznick,deLoeraSantos}.
Having  positive coefficients for $(t_1+ \dots + t_n)^N f$ certifies that $f \in \cC_{n,d}$.
However, such a Pólya representation does not imply that $f$ is strictly copositive, or equivalently that $f \in \interior(\cC_{n,d})$.
The question of which copositive polynomials admit a Pólya representation has been studied in several works \cite{PowersReznick2006,CastlePowersReznick2009,MokTo,Burgdorf2012}, and a full characterization was given in~\cite{CastlePowersReznick}.

In this paper, we take a sparse approach to copositivity and consider the \emph{sparse copositive cone} $\cC_{k \cdot A}$, defined as the cone of nonnegative polynomials on $\RR^n_{>0}$ whose support is contained in the $k$-fold Minkowski sum $k\cdot A= A + \dots +A$ for a fixed finite set $A \subseteq \ZZ^n$. To derive a Pólya-type certificate for such sparse polynomials, we apply methods from (positive) toric geometry and interpret the polynomials as homogeneous forms on the affine cone $Y_{\hat{A}} \subseteq \CC^m$ over the projective toric variety $X_A \subseteq \mathbb{P}^{m-1}$ associated to the set $A$ (see Section~\ref{Sec:Background} for more details). We say that $f$ is \emph{strictly $A$-copositive} if the induced function is positive on $Y_{\hat{A}} \cap \RR^m_{\geq0}\setminus\{0\}$, or equivalently if $f$ lies in the interior of the sparse copositive cone $\cC_{k\cdot A}$.
Our first main result says that, by modifying the multiplier in Pólya’s theorem, one can certify strict $A$-copositivity.
\begin{theoremintro}[{see Theorem~\ref{Thm:SparsePolya}}]
\label{Thm:Main}
Let $A \subseteq \ZZ^n$ be a finite set and $k \in \NN$.
Let $f$ be a (Laurent) polynomial whose support is contained in $k \cdot A$.
Then $f$ is strictly $A$-copositive if and only if there exists $N \in \NN$ such that $(\sum_{a \in A}t^a)^N f$ has nonnegative coefficients and $\Newt(f) = \conv(k \cdot A)$.
\end{theoremintro}

When $A$ consists of the standard basis vectors in $\RR^n$, the only if part of Theorem~\ref{Thm:Main} specializes to the classical version of Pólya’s theorem. In fact, to prove this result, we generalize the proof of Pólya's Theorem given in \cite[Theorem 5.5.1]{Marshall}, \cite[Theorem 5.4.1]{Scheiderer2024} and apply the Representation Theorem by Krivine~\cite{Krivine} combined with results from toric geometry.

Similar generalizations of Pólya's theorem have appeared in the context of exponential sums. 
Note that questions about nonnegativity of exponential sums of the form $\sum_{a \in A} c_{a} \exp(a \cdot z)$  can be translated into questions about the nonnegativity of the corresponding polynomial  $\sum_{a \in A} c_{a} t^a$  (possibly with real exponents) in the positive real orthant, via a logarithmic change of variables.
For example, by \cite[Theorem 4.1]{ChandrasekaranShah}, we have the following result.
Let $A \subseteq \QQ^n$ be such that $\conv(A)$ is a simplex, and every $a \in A$ is either a vertex of the simplex or lies in its interior.
Then, for every $f = \sum_{a \in A} c_{a} t^a$ that is positive on $\RR^n_{>0}$, there exists an $N \in \NN$ such that $(\sum_{a \in A}t^a)^Nf$ admits a SAGE decomposition -- that is, it can be written as a sum of nonnegative polynomials, each with at most one negative coefficient.
Theorem~\ref{Thm:Main} generalizes this result by removing the combinatorial assumptions on $\conv(A)$ and strengthening the conclusion.

Similar results appear in \cite[Theorem 7]{WangJainiYuPoupart} and \cite[Theorem 4.1]{DresslerMurray} for nonnegativity of polynomials on relatively compact sets $X \subseteq \RR^n_{>0}$. Notably, \cite[Theorem 4.1]{DresslerMurray} also holds for polynomials with real exponents, i.e., $A \subseteq \RR^n$. 
While Theorem~\ref{Thm:Main} can be directly extended to polynomials with rational exponents, it was shown in~\cite{Delzell} that a Pólya-type certificate does not exists for polynomials with irrational exponents.

In the second part of the paper, we investigate a new approach for certifying copositivity, based on the \emph{Cox ring} of the toric variety $X_A$, as an alternative to Theorem~\ref{Thm:Main}. 
We show that $f$ is strictly $A$-copositive if and only if its \emph{Cox homogenization} $f_\cox$ is nonnegative on the nonnegative real orthant and its zeros are given precisely by the \emph{irrelevant ideal} (Proposition~\ref{prop:interiorCAcox}).
The variables of $f_\cox$ correspond to the rays in the inner normal fan $\Sigma_A$ of $\conv(A)$.
Under the assumption that $\conv(A)$ is a product of simplices, we show that both the \emph{primitive collection of rays} and the generators of the irrelevant ideal give rise to a Pólya-type certificate of copositivity.
\begin{theoremintro}[{see Theorem~\ref{prop:polyaInCox_standard}}]
\label{Prop:Intro} 
   Let $A \subseteq \ZZ^n$ be a finite set such that $\conv(A) = \Delta_1 \times \dots \times \Delta_k$ is a product of full-dimensional dilated standard simplices. Denote by $C_1,\dots,C_k \subseteq [r]$ the primitive collections of rays in $\Sigma_A$, and write $\Sigma_A(n)$ for the set of its maximal cones. For a (Laurent) polynomial $f$, whose support is contained in $A$, the following are equivalent:
\begin{enumerate}
    \item[(i)] $f$ is strictly $A$-copositive;
    \item[(ii)]  $\exists N_1,\ldots,N_k\in \NN\colon\,\prod_{j=1}^k\left(\sum_{i\in C_j}x_i\right)^{N_j }f_\cox \in \RR_{\geq0}[x_1,\dots,x_r]$ and $\Newt(f) = \conv(A)$;
    \item[(iii)] $\exists N\in \NN\colon \, \left(  \sum_{\sigma \in \Sigma_A(n)} \prod_{i \notin \sigma} x_i  \right)^N f_\cox  \in \RR_{\geq0}[x_1,\dots,x_r]$ and $\Newt(f) = \conv(A)$.
\end{enumerate}
\end{theoremintro}
The same result, with a slightly more technical statement in terms of primitive collections and generators of the irrelevant ideal, holds true more generally 
for a product of arbitrary (not necessarily standard) simplices, see Theorem~\ref{thm:polyaInCox}. Moving from products of simplices to arbitrary polytopes, the situation becomes more complicated. In Section~\ref{sec:Examples}, we discuss several examples where neither (ii) nor (iii) in Theorem~\ref{Prop:Intro} can be used to certify copositivity.

\medskip

\paragraph{\textbf{Applications of copositivity.}} Our motivation for extending Pólya’s method to sparse polynomials comes from the fact that certifying the positivity of a polynomial on $\RR^n_{>0}$ has several applications across the sciences.
For example, copositivity plays an important role in studying multistationarity properties of biochemical reaction networks. In \cite[Corollary 1]{ConradiFeliuWiuf}, the authors show how to associate a polynomial $F$, called the critical polynomial, to a reaction network, such that copositivity of $F$ precludes multistationarity in the network.

In this paper, we focus on an application of copositivity in particle physics, more precisely in the study of Feynman integrals. In this setting, possible interactions in a scattering process are represented by Feynman diagrams. Loosely speaking, a Feynman integral computes the probability amplitude for these interactions to occur. The integrands of Feynman integrals are rational functions constructed from Symanzik polynomials, whose structure is determined by the combinatorics of the underlying Feynman diagram. 
By \cite[Theorem 3]{Borinsky}, the convergence of the integral can be guaranteed if the second Symanzik polynomial $\mathcal{F}$ is strictly $A$-copositive with respect to its support.
Assuming all particles are massive, this condition holds if and only if $(t_1 + \dots + t_n)^N\mathcal{F}$ has positive coefficients for some $N \in \NN$ \cite[Theorem 6.1]{SturmfelsTelek}.
The proof of this result relies on Pólya’s theorem with zeros \cite[Theorem 2]{CastlePowersReznick}, as well as the combinatorial structure of the Newton polytope of $\mathcal{F}$.
When some of the particles are massless, which is a physically relevant scenario, the Newton polytope of $\mathcal{F}$ changes, and \cite[Theorem 6.1]{SturmfelsTelek} is silent. In Section \ref{sec:Feynman} we discuss how Theorem~ \ref{Thm:Main} provides a method for making the copositivity of Symanzik polynomials manifest even in the massless setting. 

\medskip

\paragraph{\textbf{Organization. }}The paper is organized as follows. In Section~\ref{Sec:Background}, we review background on sparse copositivity and its connection to toric geometry. In Section~\ref{sec:SparsePolya}, we discuss the Representation Theorem and prove Theorem~\ref{Thm:Main}. Section~\ref{sec:Cox} focuses on Cox coordinates and contains the proof of Theorem~\ref{thm:polyaInCox}, a generalization of Theorem~\ref{Prop:Intro}. Section~\ref{sec:Examples} presents several examples comparing the Pólya-type certificates obtained from Theorem~\ref{Thm:Main} and Theorem~\ref{thm:polyaInCox}, as well as examples demonstrating that Theorem~\ref{thm:polyaInCox} fails without additional assumptions on the support set of the copositive polynomial. Finally, Section~\ref{sec:Applications} discusses Symanzik polynomials and how Theorem~\ref{Thm:Main} can be applied in this setting.

\medskip

\paragraph{\textbf{Notation. }}
We write $\mathbb{R}_{>0}^n$ and $\mathbb{R}^n_{\geq0}$ for the positive and nonnegative orthant in $\mathbb{R}^n$ respectively. For variables $t_1, \dots , t_n$ and $a \in \mathbb{Z}^n$, we use the shorthand notation $t^a = t_1^{a_1}\dots t_n^{a_n}$.
For two vectors $v,w \in \RR^n$, we denote their standard Euclidean scalar product by $v \cdot w \coloneqq v_1 w_1 + \dots + v_n w_n$.
Furthermore, we write $[n]$ for the set $\{1,\dots,n\}$.

\medskip

\paragraph{\textbf{Acknowledgments. }} We thank Simon Telen for his suggestion of looking at positivity representations in the Cox Ring. L. Baldi was funded by the Humboldt Research Fellowship for postdoctoral researchers. J. Weigert was supported by the SPP 2458 “Combinatorial Synergies”, funded by the Deutsche Forschungsgemeinschaft (DFG, German Research Foundation),
project ID: 539677510. M.~L. Telek was funded by the European Union under the Grant Agreement no. 101202522 --- POSSIS. Views and opinions expressed are however those of the author(s) only and do not necessarily reflect those of the European Union or European Research Executive Agency (REA). Neither the European Union nor the granting authority can be held responsible for them.

\section{Sparse copositivity}
\label{Sec:Background}
\subsection{Sparse copositive polynomials}
The adjective \emph{copositive} was coined by Motzkin in 1952 to describe quadratic forms that are nonnegative over the nonnegative real orthant \cite{Motzkin}.
The name was later extended to homogeneous polynomials of arbitrary degree. 
We denote by
\begin{align}
    \label{Eq:DefCnd}
 \mathcal{C}_{n,d} \coloneqq \left\{\, f \in \mathbb{R}[t_1, \dots , t_n]_d \, \mid \,  \text{ for all } t \in \mathbb{R}^n_{\geq0}, \ f(t) \geq 0 \,\right\},
\end{align}
the \emph{cone of copositive homogeneous polynomials}. Here, $\RR[t_1, \dots , t_n]_d$ denotes the space of homogeneous polynomials in $n$ variables of degree $d$.
The interior of $\mathcal{C}_{n,d}$ consists of homogeneous polynomials $f$ that are \emph{strictly copositive}, that is, $f(t) > 0$ for all $t \in \mathbb{R}^n_{\geq0} \setminus \{ 0\}$.
To certify whether a polynomial lies in $\interior \mathcal{C}_{n,d}$, one might use the following classical theorem by Pólya.
\begin{theorem}[Pólya's theorem \cite{Polya}]
\label{Thm:ClassicalPolya}
    Let $f \in \mathbb{R}[t_1, \dots , t_n]_d$.
    If $f(x) > 0$ for all $x \in \mathbb{R}^n_{\geq0} \setminus \{ 0\}$
    then there exists $N \in \mathbb{N}$ such that $(t_1 + \dots + t_n)^Nf$ has only positive coefficients.
\end{theorem}
The main goal of this work is to provide a certificate for copositivity, similar to Theorem~\ref{Thm:ClassicalPolya}, for polynomials with a fixed set of exponent vectors.
For a Laurent polynomial $f = \sum_{a \in \mathbb{Z}^n} c_a t^a \in \mathbb{R}[t_1^\pm, \dots, t_n^\pm]$, we write $\supp(f) := \{ \, a \in \mathbb{Z}^n \, \mid \, c_a \neq 0 \, \}$ for the \emph{support} of $f$.
For a finite set $A \subseteq \mathbb{Z}^n$, we denote by
\begin{align}
    \RR[t_1^\pm,\ldots,t_n^\pm]_A := \big\{ f  \in \mathbb{R}[t_1^\pm,\dots,t_n^\pm] \, \mid \, \supp(f) \subseteq A \big\}
\end{align}
the vector space of Laurent polynomials whose support is contained in $A$. 
Moreover, we  define the \emph{sparse copositive cone} as
\begin{align}
\label{Eq:DefSparseCoposCone}
\mathcal{C}_A \coloneqq \left\{ f \in \RR[t_1^\pm,\ldots,t_n^\pm]_A \mid \text{ for all } t \in \mathbb{R}^n_{>0},\ f(t) \geq 0 \right\}.    
\end{align}

In contrast to the definition of the cone $\mathcal{C}_{n,d}$ in \eqref{Eq:DefCnd}, the definition of the sparse copositive cone~$\mathcal{C}_{A}$ allows negative integer exponents and does not require homogeneity. To justify the use of the term copositive for $\mathcal{C}_{A}$, we recall the following simple fact.

\begin{lemma}
    Let $A \subseteq \mathbb{Z}^n$ be a finite set and $f \in \RR[t_1^\pm,\ldots,t_n^\pm]_A$. Let $a_* \in \mathbb{Z}^n$ be such that $a_* + A \subseteq \mathbb{N}^n$ and denote $\widetilde{f} \in \RR[t_0, \dots, t_n]_d$ the homogenization of $t^{a_*} f$ in $t_0$.
    Then $f$ lies in the sparse copositive cone~$\mathcal{C}_A$ if and only if $\widetilde{f} \in \mathcal{C}_{n+1,d}$.
\end{lemma}

\begin{proof}
   Since a continuous function is nonnegative on $\mathbb{R}_{>0}^{n+1}$ if and only if it is nonnegative on $\mathbb{R}_{\geq0}^{n+1}$, we have 
    \[ \mathcal{C}_{n+1,d} = \left\{\, g \in \mathbb{R}[t_0, \dots , t_n]_d \, \mid \,  \text{ for all } t \in \mathbb{R}^{n+1}_{> 0}, g(t) \geq 0 \,\right\}.\]
    Now, the statement follows, since $f(t) \geq 0$ for $t = (t_1, \dots, t_n) \in \RR^n_{>0}$ if and only if $\widetilde{f}(1,t_1,\dots ,t_n) =  t^{a_*} f(t) \geq 0$, which is equivalent to $\widetilde{f}(t_0,t_1,\dots ,t_n)  \geq 0$ for any $(t_0,t_1,\dots,t_n) \in \mathbb{R}_{>0}^{n+1}$.
\end{proof}

\begin{example}
\label{Ex:Running}
We conclude this subsection with an example that illustrates the differences between homogeneous and sparse copositive polynomials.
A simple computation shows that the homogeneous polynomial
    \begin{align*}
         h &=t_1^3 + t_2^3 +t_3^3 +t_4^3+t_1t_4^2-1.9t_2t_4^2+t_2^2t_4+  t_3t_4^2 + t_1t_3t_4-1.9t_2t_3t_4+t_2^2t_3 \\
         &=t_1^3 + t_2^3 +t_3^3  + t_4 (t_4 - t_2)^2 + 0.1t_2t_4^2 + t_3 (t_4 - t_2)^2 + 0.1t_2t_3t_4 +t_1t_4^2 + t_1t_3t_4
    \end{align*}
    takes only positive values on $\mathbb{R}^4_{\geq 0} \setminus \{0\}$, which implies that $h \in \interior \mathcal{C}_{4,3}$. To certify copositivity of $h$, one can apply Pólya's theorem, which guarantees that $(t_1 + t_2 + t_3 +t_4)^Nh$ has positive coefficients for some $N \in \mathbb{N}$. In this example, the smallest such $N$ is $11$.

We slightly modify $h$ by removing the monomials $t_1^3,t_2^3,t_3^3$, obtaining the polynomial
\begin{equation}
     \label{Eq:Running}
  \begin{aligned}
         f &= t_4^3+t_1t_4^2-1.9t_2t_4^2+ t_2^2t_4  +t_3t_4^2 + t_1t_3t_4-1.9t_2t_3t_4+t_2^2t_3 \\
         &=t_4 (t_4 - t_2)^2 + 0.1t_2t_4^2 + t_3 (t_4 - t_2)^2 + 0.1t_2t_3t_4 +t_1t_4^2 + t_1t_3t_4.
    \end{aligned}
    \end{equation}
Since too many degree $3$ monomials are missing (the polynomial $f$ is too sparse), the Newton polytope of $f$ is not the $3$-dilated simplex that one would expect from a homogeneous polynomial of degree $3$. For an illustration, we refer to Figure~\ref{FIG1:NewtonPoly}. Even though $f$ is positive on $\mathbb{R}^4_{>0}$, 
the sparsity of $f$ causes zeros on the boundary of $\mathbb{R}^4_{\geq0}$: for example,
we have $f(t_1,0,t_3,0) = 0$ for any $t_1,t_3 \geq 0$. Thus $f \notin \interior \mathcal{C}_{4,3}$ and Theorem~\ref{Thm:ClassicalPolya} does not guarantee the existence of a Pólya certificate. In fact the coefficient of the monomial $t_2t_3^{N+1}t_4$ in the polynomial $(t_1 + t_2+ t_3+t_4)^Nf$ is negative and equal to $-1.9$ for any $N\in \NN$.

\end{example}

\begin{figure}[t]
\centering
\includegraphics[scale=0.4]{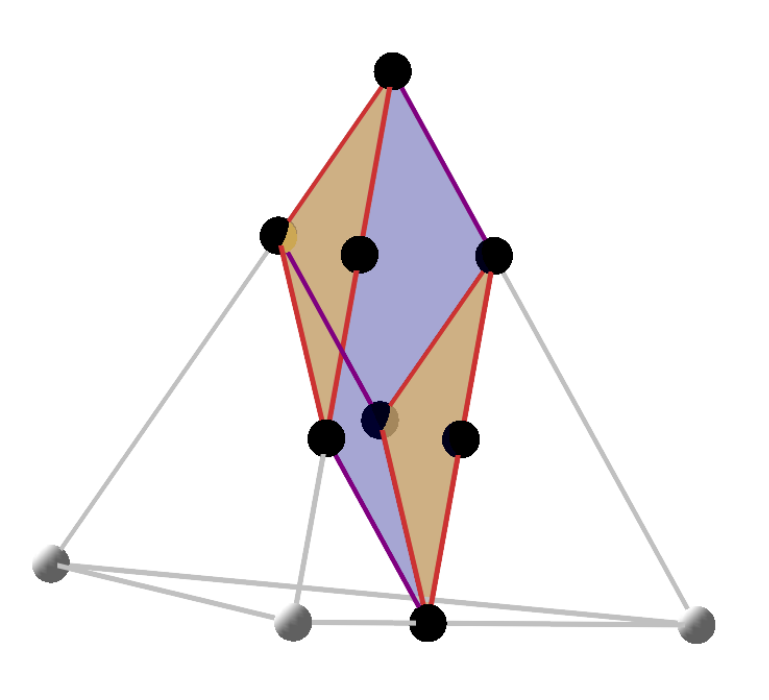}
\caption{{\small  Newton polytope of the polynomial $f$ from \eqref{Eq:Running} sitting inside the $3$-dilated simplex.}}\label{FIG1:NewtonPoly}
\end{figure}

\subsection{Positive toric geometry}
\label{Sec:PosToricGeom}
The goal of this subsection is to define a suitable compactification of $\mathbb{R}^n_{>0}$ and to describe the interior of the sparse copositive cone given in \eqref{Eq:DefSparseCoposCone}.
Even though some versions of this latter result may exist in the literature or be known to experts, we are not aware of a reference. Hence, we provide full arguments in Proposition~\ref{Prop:InteriorOfCA}, which will play a crucial role in the proof of Theorem~\ref{Thm:Main}.

We first recall some notions from toric geometry, following the notation in \cite{telen2022introductiontoricgeometry,TelenBook}.
For a finite set $A = \{ a_1, \dots , a_m\} \subseteq \mathbb{Z}^n$, we consider the monomial maps
\begin{equation}
    \label{Eq:MonomialMap}
\begin{aligned}
\phi_A &\colon (\mathbb{C}^*)^n \to \mathbb{P}^{m-1}, \quad t=(t_1,\dots,t_n) \mapsto [t^{a_1}\colon \dots \colon t^{a_m}],\\
\varphi_{\hat{A}}&\colon (\mathbb{C}^*)^{n+1} \to \mathbb{C}^{m}, \quad (u,t) = (u,t_1,\dots,t_n) \mapsto (u t^{a_1}, \dots , u t^{a_m}).
\end{aligned}
\end{equation}
The Zariski closure of the image of $\phi_A$ in $\mathbb{P}^{m-1}$ is called the \emph{projective toric variety} $X_A$.
Its affine cone is the \emph{affine toric variety} $Y_{\hat{A}}$, which is the Zariski closure of the image of $\varphi_{\hat{A}}$ in $\mathbb{C}^m$, see e.g. \cite[Proposition 3.7]{telen2022introductiontoricgeometry}.
Hence $\CC[X_A] = \CC[Y_{\hat{A}}]$.
The homogeneous coordinate ring of $X_A$ is more precisely the graded ring
\begin{align}
\label{Eq:CoordinateRingXA}
    \mathbb{C}[X_A] = \mathbb{C}[t^{a_1}, \dots , t^{a_m}] = \bigoplus_{k=0}^\infty \bigoplus_{ \substack{b = \sum_{i=1}^m v_ia_i \\ v_1, \dots, v_m \in \mathbb{N}, \sum v_i = k }} \mathbb{C} \; t^b.
\end{align}

We are interested in the positive part of the projective toric variety $X_A$.
To that end, we consider the \emph{nonnegative} and \emph{positive projective orthants}
\begin{align*}
    \mathbb{P}^{m-1}_{\geq 0} &\coloneqq \{  [x_1 \colon \dots \colon x_m] \in \mathbb{P}^{m-1}(\RR) \, \mid \, x_1, \dots, x_m \in \mathbb{R}_{\geq0}  \},\\
    \mathbb{P}^{m-1}_{> 0}  & \coloneqq \{  [x_1 \colon \dots \colon x_m] \in \mathbb{P}^{m-1}(\RR) \, \mid \, x_1, \dots, x_m \in \mathbb{R}_{>0}  \}.
\end{align*}
Using the map $\PP^{m-1}_{\geq 0} \to \RR^m, \; x \mapsto \tfrac{1}{x_1+ \dots +x_m}(x_1, \dots, x_m)$, we see that $\mathbb{P}^{m-1}_{\geq 0} $ is homeomorphic to the $m-1$-dimensional probability simplex in $\RR^m$, and $\mathbb{P}_{>0}^{m-1} = \interior \PP^{m-1}_{\geq 0}$ to its relative interior.

The \emph{positive part} of a projective toric variety $X_A$ is defined as $(X_A)_{>0} \coloneqq X_A \cap \mathbb{P}_{>0}^{m-1}$ and the \emph{nonnegative part} as $(X_A)_{\ge 0} \coloneqq X_A \cap \mathbb{P}_{\ge 0}^{m-1}$. The nonnegative part $(X_A)_{\geq0}$ is also the Euclidean closure of $(X_A)_{>0}$ inside $\mathbb{P}_{\geq0}^{m-1}$. 
This follows from the fact that the algebraic moment map is a homeomorphism when restricted to the nonnegative part, see e.g. \cite[Th.~12.5.2]{CoxLittleSchenck}. The positive part $(X_A)_{>0}$ can alternatively be described as the image of $\mathbb{R}^n_{>0}$ under the monomial map $\phi_{A}$ in~\eqref{Eq:MonomialMap}. 
The nonnegative part has a similar description in terms of faces of the polytope $\conv(A)$, see \cite[Section 5.1]{TelenBook}. 
We also notice that the cone over $(X_A)_{>0}$ (respectively, $(X_A)_{\ge 0}$) is equal to $(Y_{\hat{A}})_{>0} \coloneqq Y_{\hat{A}} \cap \RR_{>0}^m$ (respectively $(Y_{\hat{A}})_{\geq 0} \coloneqq Y_{\hat{A}} \cap \RR_{\geq 0}^m$).

Since $\CC[t^{a_1}, \dots, t^{a_m}] = \CC[Y_{\hat{A}}]$
every polynomial $f \in \CC[t^{a_1}, \dots, t^{a_m}]$ induces a function on the affine toric variety~$Y_{\hat{A}}$.
In the proof of Proposition~\ref{Prop:InteriorOfCA} and of Theorem~\ref{Thm:Main} we will need the precise construction, so we discuss this correspondence in more detail.
For $k \in \mathbb{N}$, we write
\begin{align}
\label{Eq:MinkowskiAk}
 k \cdot A := \underbrace{A + \dots + A}_{k \text{ times}} = \big\{ \, b \in \mathbb{Z}^n \, \mid \, \exists v_{1 b}, \dots, v_{m b} \in \mathbb{N}\colon b = \sum_{j=1}^m v_{j b} a_j \text{ and }  \sum_{j=1}^m v_{j b} = k \big\} 
\end{align}
for the $k$-times Minkowski sum of $A$. Notice that in general both inclusions in $\{ka\,|\,a\in A\}\subsetneq k \cdot A\subsetneq \{kx\,|\,x\in \conv(A)\}\cap \ZZ^d$ can be strict.
Write now $y_i \coloneqq u t^{a_i}$. If $f = \sum_{b \in k \cdot A} c_b t^b \in \RR[t_1^{\pm 1}, \dots , t_n^{\pm 1}]_{k \cdot A}$ then by definition there exists a (non necessarily unique) integer matrix $V = (v_{ij})$ such that
\begin{align}
\label{eq:DefV}
    u^k f(t_1, \dots , t_n) = \sum_{b \in k\cdot A} c_b y_1^{v_{1 b}} \dots y_m^{v_{m b}}.
\end{align}
We will write $\hat{f}_V = f(y_1, \dots , y_m) \in \RR[y_1, \dots , y_m]$ for such a polynomial depending on $V$. Of course, the value of $\hat{f}_V$ on $Y_{\hat{A}} \subseteq \RR^m$ only depends on $f$ and not on the choice of the matrix $V$.

\begin{example}
\label{ex:Vmatrix}
Let $A = \{ \, (0,0), \, (1,0), \, (0,1), \, (1,1) \, \}$. For $k = 2$, we have
\begin{align*}
    2 \cdot A = \{ \, (0,0), \, (1,0), \, (2,0), \, (0,1), \, (1,1), \, (2,1),  (0,2), \, (1,2), \, (2,2) \, \}.
\end{align*}
Fix a Laurent polynomial $f=\sum_{b\in k\cdot A}c_bt^b\in \RR[t_1^\pm,\ldots,t_n^\pm]_{k\cdot A}$. 
To explicitly write the function $\hat{f}_V$ we express every point in $k\cdot A$ as a sum of $k$ points in $A$ and record these decompositions in a matrix $V$.
There are two different ways to express $(1,1)$ as sum of two elements of $A$: either $(1,1) = (0,0) + (1,1)$, or $(1,1) = (1,0) + (0,1)$.
  These give rise to two choices for the $V$ matrix
  \begin{align*}
   V_1 = \begin{pmatrix}
       2 & 1 & 0 & 1 & \mathbf{1} & 0 & 0 & 0& 0 \\
       0 & 1 & 2 & 0 & \mathbf{0} & 1 & 0 & 0& 0  \\
       0 & 0 & 0 & 1 & \mathbf{0} & 0 & 2 & 1 & 0 \\
       0  & 0 & 0 & 0 & \mathbf{1} & 1 & 0 & 1& 2 
   \end{pmatrix}, \quad    V_2 = \begin{pmatrix}
       2 & 1 & 0 & 1 & \mathbf{0} & 0 & 0 & 0& 0 \\
       0 & 1 & 2 & 0 & \mathbf{1} & 1 & 0 & 0& 0  \\
       0 & 0 & 0 & 1 & \mathbf{1} & 0 & 2 & 1 & 0 \\
       0  & 0 & 0 & 0 & \mathbf{0} & 1 & 0 & 1& 2 
   \end{pmatrix}.
\end{align*}
The two polynomials $\hat{f}_{V_1}$, $\hat{f}_{V_2} \in \RR[y_1, \dots , y_m]$ are then different, but they agree on $Y_{\hat{A}}$.
\end{example}

The restriction of the monomial map $\varphi_{\hat{A}}$ to $\RR^{n+1}_{>0}$ is surjective on $(Y_{\hat{A}})_{>0}$ \cite[Proposition 5.1.3]{TelenBook}. 
For $f \in \RR[t_1^\pm, \dots , t_n^\pm]_{k\cdot A}$, we summarize the relevant maps in the following commutative diagram
\begin{equation}
     \label{eq:ftildeCD}
                  \begin{tikzcd}
\RR_{>0}^{n+1}\arrow[rr,two heads,"\varphi_{\hat{A}}"]\arrow[d,"u^k \cdot f"] & & (Y_{\hat{A}})_{>0}\arrow[d,hookrightarrow] \\
           \RR  & & \arrow[ll,"f = \hat{f}_V"] (Y_{\hat{A}})_{\geq0}
         \end{tikzcd}
     \end{equation}
\begin{definition}
    We say that $f \in \RR[t_1^\pm,\ldots,t_n^\pm]_{k \cdot A}$ is \emph{$A$-copositive} (respectively, \emph{strictly $A$-copositive}) if $f \ge 0$ on $(Y_{\hat{A}})_{\geq0}$ (respectively, if $f > 0$ on $(Y_{\hat{A}})_{\geq0} \setminus \{ 0 \}$).
\end{definition}
In the above definition, if we set $A = \{\,e_1,  \dots , e_n\, \}$ (where $e_i$ is the standard $i$-th basis vector) we recover the usual notion of (strict) copositivity for homogeneous polynomials.
Strict $A$-copositivity is an equivalent condition to $f \in \interior \mathcal{C}_{k\cdot A}$, as we show in the following proposition.

\begin{proposition}
\label{Prop:InteriorOfCA}
Let $A \subseteq \mathbb{Z}^n$ be a finite set, $k \in \mathbb{N}$ and $f \in \RR[t_1^{\pm 1}, \dots , t_n^{\pm 1}]_{k \cdot A}$. 
Then the following are equivalent.
\begin{itemize}
    \item[(i)] $f$ belongs to $\interior \mathcal{C}_{k\cdot A}$, the interior of the sparse copositive cone.
    \item[(ii)] $f$ is strictly  $A$-copositive.
\end{itemize}
\end{proposition}
\begin{proof}
Using  the surjection $\mathbb{R}^{n+1}_{>0} \twoheadrightarrow (Y_{\hat{A}})_{>0}$ and the commutative diagramm (\ref{eq:ftildeCD}), we have that $f$ is nonnegative on $\mathbb{R}^n_{>0}$, if and only if it is nonnegative on $(Y_{\hat{A}})_{>0}$. By continuity, it follows that
\begin{align}
\label{Eq:CkA_correspondence}
\mathcal{C}_{k \cdot A} = \{\, f \in \mathbb{R}[t_1^\pm,\ldots,t_n^\pm]_{k \cdot A} \, \mid \, \forall y \in (Y_{\hat{A}})_{\ge 0}, \ f(y) \geq 0  \,\}.
\end{align}
The convex cone $(Y_{\hat{A}})_{\ge 0}$ has a compact section, (e.g. given by $y_1 + \dots + y_m = 1$,) which we identify with  $(X_{A})_{\ge 0} \subseteq (\PP^{m-1})_{\ge 0}$. For $f \in \mathbb{R}[t_1^\pm,\ldots,t_n^\pm]_{k \cdot A}$, $y \in (Y_{\hat{A}})_{\ge 0} \setminus \{0\}$ and $\lambda \in \RR_{>0}$, the sign of $f(\lambda y) = \lambda^k f(y)$ is constant in $\lambda$ and therefore determined only by its value on $(X_{A})_{\ge 0}$. 
Assume first that $f(y)>0$ for all $y\in (X_A)_{\geq 0}$, then since  $(X_A)_{\geq 0}$ is compact, any small perturbation of $f$ remains strictly positive on $(X_A)_{\geq 0}$, proving that $f\in \interior(\mathcal{C}_{k\cdot A})$. On the other hand if for $f\in \mathcal{C}_{k\cdot A}$ there exists $y\in (X_A)_{\geq 0}$ with $f(y)=0$, then for each $\epsilon>0$ we find that $f-\epsilon (\sum_i t^{a_i})^k$ is negative at $y$, proving that $f$ is not in the interior of $\mathcal{C}_{k\cdot A}$ in this case.\end{proof}

We introduce a notation following e.g. \cite{Bihan2002}, necessary to give another characterization of strict $A$-copositivity.
For $f = \sum_{b \in \ZZ^n} c_b t^b \in \RR[t_1^{\pm 1}, \dots , t_1^{\pm 1}]$ and $S \subseteq \RR^n$, we write
    \[
        f^S \coloneqq \sum_{b \in S \cap \ZZ^n} c_b t^b
    \]
    for the \emph{truncation} of $f$ to $S$.
If $S = F_w$ is a face of the Newton polytope of $f$ with inner normal vector $w \in \mathbb{R}^n$, then $f^{F_w}$ is also called the \emph{initial form} of $f$ with respect to $w$, and is denoted by $\initial_w(f)$. By \cite[Lemma~2.6.2]{MaclaganSturmfels}, initial forms are multiplicative, that is, for every $f,g \in \mathbb{R}[t_1^\pm, \dots , t_n^\pm]$ and $w \in \mathbb{R}^n$ we have 
\begin{align}
\label{Eq:InitialFormsFactor}
    \initial_w(fg) = \initial_w(f) \initial_w(g).
\end{align}
This crucial property of initial forms will be used in the proof of Theorem~\ref{Thm:Main}.
We can now state another equivalent characterizations of $\interior \mathcal{C}_{k\cdot A}$.

\begin{proposition}
\label{Prop:InteriorOfCA_bis}
Let $A \subseteq \mathbb{Z}^n$ be a finite set, $k \in \mathbb{N}$ and $f \in \RR[t_1^{\pm 1}, \dots , t_n^{\pm 1}]_{k \cdot A}$. 
Then the following are equivalent.
\begin{itemize}
    \item[(ii)] $f$ is strictly  $A$-copositive,
    \item[(iii)]  $f^{k\cdot(F \cap A)}(t) > 0$ for all $t \in \mathbb{R}^n_{>0}$ and all nonempty faces $F \subseteq \conv(A)$,
    \item[(iv)]  $f^G(t) > 0$ for all $t \in \mathbb{R}^n_{>0}$ and all nonempty faces $G \subseteq \conv(k \cdot A)$.
\end{itemize}
\end{proposition}
\begin{proof}
The proof relies on the following observations.
For every $y = (y_a \colon a \in A) \in (Y_{\hat{A}})_{\geq0} \setminus \{0\}$, there exists a nonempty face $F$ of the convex hull of $A$  and $(u,t) \in \mathbb{R}^{n+1}_{>0}$ such that
\begin{align}
\label{Eq:FacialSets}
    y_a = u t^{a} \quad \text{ if } a \in F, \qquad \text{ and } \qquad  y_a = 0 \quad \text{ if } a \notin F.
\end{align}
Conversely, for every nonempty face $F \subseteq \conv(A)$ and $(u,t) \in \RR_{>0}^{n+1}$ there exists a point $y = (y_a \colon a \in A) \in (Y_{\hat{A}})_{\geq0}$ such that \eqref{Eq:FacialSets} holds true. These claims follow from the discussion in \cite[Section 5.1]{TelenBook}. For the sake of completeness, we recall the arguments here.

By \cite[Corollary 4.1.3]{TelenBook} the affine toric variety $Y_{\hat{A}}$ has the stratifiaction
\[ Y_{\hat{A}} = \bigsqcup_{\substack{F \subseteq \conv(A)\\ F \neq \emptyset}} Y_{\hat{A},F}^\circ\]
where $Y_{\hat{A},F}^\circ = \{\, y \in Y_{\hat{A}} \, \mid \, y_a \neq 0 \text{ if and only if }  a \in F \cap A\,\}$. Intersecting with $\RR^m_{\geq0}$ gives
\begin{align}
\label{eq:proofToric1}
  ( Y_{\hat{A}})_{\geq 0} =  Y_{\hat{A}} \cap  \RR^m_{\geq0}  = \bigsqcup_{\substack{F \subseteq \conv(A)\\ F \neq \emptyset}} 
 \Big( Y_{\hat{A},F}^\circ  \cap  \RR^m_{\geq0} \big).
 \end{align}
By \cite[Proposition 4.1.5]{TelenBook}, the map
\begin{align}
\label{eq:proofToric2} \varphi_{A \cap F} \colon \; (\CC^*)^{n+1} \to Y_{\hat{A},F}^\circ, \qquad (u,t) \mapsto  \varphi_{A \cap F}(u,t) =  \begin{cases}
    u t^{a} & \text{ if } a \in A\cap F \\
    0  & \text{ if } a \notin A\cap F 
\end{cases} 
\end{align}
is well-defined and surjective. This gives the second claim above, that is, for every nonempty face $F \subseteq \conv(A)$ and $(u,t) \in \RR^{n+1}$, we have $\varphi_{A\cap F}(u,t) \in (Y_{\hat{A}})_{\geq 0}$.
To prove the first claim, let $y \in (Y_{\hat{A}})_{\geq 0}$. By \eqref{eq:proofToric1} and \eqref{eq:proofToric2}, there exists a nonempty face $F\subseteq \conv(A)$ and $(u,t) \in (\CC^*)^{n+1}$ such that $\varphi_{A \cap F} (u,t) = y$. Taking absolute values we also find a preimage of $y$ with positive entries, since  $y = |y| = |\varphi_{A \cap F} (u,t)| = \varphi_{A \cap F} (|u|,|t|)$.

We now prove the equivalence of  (ii) and (iii).
Let $y \in (Y_{\hat{A}})_{\geq 0}, \, F \subseteq \conv(A)$ and $(u,t) \in \RR^{n+1}_{\geq 0}$ satisfying \eqref{Eq:FacialSets}.
For a fixed matrix $V = (v_{ab})$ as in \eqref{eq:DefV}, we have the equality
\[
    u^kf(t) = \hat{f}_V(y) = \sum_{b \in k\cdot A} c_b \prod_{a \in A} y_a^{v_{a b}} = \sum_{ {b \in k\cdot A}} c_b \prod_{a \in F \cap A} y_a^{v_{a b}} = \sum_{b \in k\cdot(F \cap A) } c_b \prod_{a \in A} y_a^{v_{a b}} = f^{k\cdot(F \cap A)}(t),
\]
where the third equality follows from \eqref{Eq:FacialSets}, and in the fourth equality we use that $v_{ab} = 0$ for all $b\in k \cdot A$ and $a \notin F$.
This shows (ii) $\Leftrightarrow$ (iii).

In the last part of the proof, we show the equivalence of (iii) and (iv). A vector lies on a face $G \subseteq\conv(k\cdot A)$ with inner normal vector $w$ if and only if it belongs to $k \cdot F$, where $F$ is the face of $\conv(A)$ with inner normal vector $w$. In particular, we have $G \cap (k \cdot A) = k \cdot (F \cap A)$, and (iii) $\Leftrightarrow$ (iv) follows.
\end{proof}

\begin{example}
\label{Ex:Running:Interior}
    We revisit the polynomial $f$ from Example~\ref{Ex:Running}. Let $A \subseteq \mathbb{Z}^4$ denote the support of $f$. First, we consider the following edges of $\conv(A)$ (cf. Figure~\ref{FIG1:NewtonPoly})
    \[F_1 =\conv(\, (0,0,0,3), \, (0,2,0,1) \,) \quad F_2 =\conv(\, (0,0,1,2), \, (0,2,1,0) \,).  \]
    Since the truncated polynomials 
      \begin{align*}
         f^{F_1} = t_4^3-1.9t_2t_4^2+ t_2^2t_4, \quad f^{F_2} = t_3t_4^2 -1.9t_2t_3t_4 + t_2^2 t_3
    \end{align*}
    are essentially quadratic and univariate, we have that  $f^{F_1}(t) > 0$ and $f^{F_2}(t) > 0$ for all $t \in \mathbb{R}^4_{>0}$. Since $F_1, F_2$ are the only faces of $\conv(A)$ that contains in its relative interior exponent vectors of $f$ whose corresponding coefficient has a negative sign and $F_1 \cap F_2 = \emptyset$, we conclude that $f^{F}(t) > 0$ for all $t \in \mathbb{R}^4_{>0}$.
    Thus, $f \in \interior \mathcal{C}_{A}$ by Proposition \ref{Prop:InteriorOfCA}.
\end{example}

\section{Pólya's theorem for sparse polynomials}
\label{sec:SparsePolya}
In this section, we prove the first extension of Pólya's theorem (Theorem~\ref{Thm:SparsePolya}). Our approach generalizes the argument given in \cite[Theorem 5.5.1]{Marshall} or \cite[Theorem~5.4.1]{Scheiderer2024}, which is used to prove the classical version of Pólya's theorem (Theorem~\ref{Thm:ClassicalPolya}). We adapt this strategy to our setting by combining it with tools from positive toric geometry (cf. Section~\ref{Sec:PosToricGeom}).

\subsection{The Representation Theorem}
\label{Sec:ReprThm}
Let $V$ be an affine $\RR$-variety and $\RR[V]$ its real coordinate ring. We denote $V(\RR) \cong \mathrm{Hom}_\RR(\RR[V], \RR)$ the set of its $\RR$-rational points.

\begin{definition}
    A subset $T \subseteq \RR[V]$ is called \emph{preprime} if $T+T \subseteq T$, $T \cdot T \subseteq T$ and $\QQ_{\ge 0} \subseteq T$.
    A preprime $T \subseteq \RR[V]$ is called \emph{Archimedean} if for all $f \in \RR[V]$ there exists $0< n \in \NN$ such that $n+f, n-f  \in T$.
\end{definition}
We will need the following lemma to prove Archimedeanity of preprimes.
\begin{lemma}[{\cite[Prop.~5.1.3]{Marshall}}]
\label{Lemma:Archimedean}
    Let $T \subseteq \RR[V]$ be a preprime. Then
    \[H_T \coloneqq  \{\,f \in \RR[V] \, \mid \, \exists n \in \mathbb{N} \colon n + f, n - f \in T\,\}\]
    is a subring of $\RR[V]$, and
    $T$ is Archimedean if and only if $H_T = \RR[V]$.
\end{lemma}

\begin{theorem}[{Representation Theorem, \cite[Th.~5.4.4]{Marshall}, \cite[Cor.~5.3.3]{Scheiderer2024}}]
    \label{Thm:RepThm}
    Let $V$ be an affine $\RR$-variety and $T \subseteq \RR[V]$ be an Archimedean preprime. If $\mathcal{K}(T) \coloneqq \{ \, x \in V(\RR) \mid f(x) \ge 0 \ \forall f \in T  \, \}$ and $g(x)>0$ for all $x \in \mathcal{K}(T)$, then $g \in T$.
\end{theorem}
We denote by $\RR_{\ge 0}[t^{a_1},\ldots,t^{a_m}]$ the set of polynomials in $t^{a_1},\ldots,t^{a_m}$ with nonnegative coefficients, i.e. the preprime generated by $\RR_{\ge 0}$ and $t^{a_1},\ldots,t^{a_m}$.
\begin{lemma}
\label{Lemma:Archimedian}
     Let $A = \{ \, a_1, \dots, a_m \, \}$. The preprime
     \[
     T = \RR_{\geq 0}[t^{a_1},\ldots,t^{a_m}]+\langle 1-\sum_{i=1}^m t^{a_i}\rangle \subseteq \RR[t^{a_1}, \dots , t^{a_m}]
     \]
     generated by $\RR_{\geq0}$, $t^{a_1}, \dots , t^{a_m}$ and by the ideal $\langle 1-\sum_{i=1}^m t^{a_i}\rangle$ 
     is Archimedean.
\end{lemma}
\begin{proof}
    By Lemma~\ref{Lemma:Archimedean}, it is enough to show that $t^a \in H_{T}$ for all $a \in A$ and $\mathbb{R} \subseteq T$. The latter inclusion follows directly from the fact that for every $s \in \mathbb{R}$ there exists $n \in \mathbb{N}$ such that $n + s \in \mathbb{R}_{\geq0} \subseteq T$. To see that  $t^a \in H_{T}$, note that $1+t^a \in T$ by construction. Moreover, $1 - t^a = \sum_{b \in A, b \neq a} t^b + \big( 1 - \sum_{b \in A} t^b \big) \in T$. Thus, we have $t^a \in H_T$ for each $a \in A$.
\end{proof}

We can lift a representation of $f\in T$ to a P\'olya-type certificate of copositivity by clearing denominators carefully. This is a key step in all our extensions to P\'olya's theorem.
\begin{proposition}
    \label{prop:polya}
    Let $T \subseteq \RR[t^{a_1}, \dots , t^{a_m}]$ be the preprime as in Lemma~\ref{Lemma:Archimedian}. If $f \in \RR[t_1^{\pm}, \dots , t_n^{\pm}]_{k\cdot A}$ is supported on $k\cdot A$ and $f \in T$, then there exist $N \in \NN$ such that \[
    \big( \sum_{i=1}^m t^{a_i} \big)^N f \in \RR_{\ge 0}[t^{a_1},\ldots,t^{a_m}].
    \]
\end{proposition}
\begin{proof}
    By hypothesis, we can write 
    \[
        f = g + h \cdot (1-\sum_{i=1}^m t^{a_i})
    \]
    for some $g =  \sum d_v (t^{a_1})^{v_1} \dots (t^{a_m})^{v_m}\in \RR_{\ge 0}[t^{a_1},\ldots,t^{a_m}]$ and $h \in \RR[t^{a_1}, \dots , t^{a_m}]$. Since $f \in \RR[t_1^{\pm}, \dots , t_n^{\pm}]_{k\cdot A}$, substituting $t^{a_i} \mapsto \tfrac{t^{a_i}}{\sum_{i=1}^m t^{a_i}}$ gives 
    \[
        \frac{f}{(\sum_{a \in A} t^a)^k} = \sum d_v \tfrac{1}{ (\sum_{a \in A} t^a)^{v_1+ \dots + v_m}} (t^{a_1})^{v_1} \dots (t^{a_m})^{v_m}
    \]
   Clearing denominators in the above equality by multiplying both sides with $(\sum_{a \in A}t_a)^N$ for $N = \max\{k , \max_{v} \{ v_1 + \dots +v_m\}\}$ yields the desired result.
\end{proof}

\subsection{First extension of Pólya's theorem}
We now prove our first main result. For $f\in \RR[t_1^\pm,\ldots,t_n^\pm]_A$ we say that $f$ has positive coefficients if for every point $a \in A$, the monomial $t^a$ appears with strictly positive coefficient.

\begin{theorem}
\label{Thm:SparsePolya}
Let $A = \{ \,a_1, \dots , a_m \,\} \subseteq \ZZ^n$ be a finite set, $k \in \NN$ and $f \in \mathbb{R}[t_1^{\pm}, \dots , t_n^{\pm}]_{k \cdot A}$. Then the following are equivalent.
\begin{enumerate}
    \item[(i)] $f$ is strictly $A$-copositive.
    \item[(ii)] $\Newt(f) = \conv(k \cdot A)$, and there exists $M \in \NN$ s.t. $(\sum_{i=1}^m t^{a_i})^M f\in \RR[t_1^\pm,\ldots,t_n^\pm]_{(k+M)\cdot A}$ has positive coefficients.
    \item[(iii)] $\Newt(f) = \conv(k \cdot A)$, and there exists $N \in \NN$ s.t. $(\sum_{i=1}^m t^{a_i})^N f\in \RR[t_1^\pm,\ldots,t_n^\pm]_{(k+N)\cdot A}$ has nonnegative coefficients.
\end{enumerate}
\end{theorem}

\begin{proof}
We first show (i) $\Rightarrow$ (ii). Let $T$ be the Archimedean preprime in Lemma~\ref{Lemma:Archimedian}. Since $f$ is strictly $A$-copositive, it is strictly positive on \[(Y_{\hat{A}})_{\geq 0} \cap \{ \,1 - \sum_{a \in A} t^a =0\, \} = \{ \, x \in Y_{\hat{A}}(\RR) \mid f(x) \ge 0 \ \forall f \in T  \, \} = \mathcal{K}(T).\]
It follows from the Representation Theorem \ref{Thm:RepThm} that $f \in T$. Proposition~\ref{prop:polya} then implies that there exist $N \in \NN$ such that 
\[
    \big( \sum_{i=1}^m t^{a_i} \big)^N f \in \RR_{\ge 0}[t^{a_1},\ldots,t^{a_m}].
\]
This means that $( \sum_{i=1}^m t^{a_i} )^N f$ has a representation as a degree $k+N$ polynomial in $t^{a_1}, \dots , t^{a_m}$ with nonnegative coefficients.

We can make all coefficients positive {by possibly increasing $N$}. Since $\mathcal{K}(T)$ is compact and $f$ is strictly positive, there exists $\varepsilon > 0$ such that $f > 2\varepsilon$ on $\mathcal{K}(T)$. Notice that $\sum_{i=1}^m t^{a_i}$ is constant and equal to one on $\mathcal{K}(T)$, and hence $f - \varepsilon (\sum_{i=1}^m t^{a_i})^k > 0$ on $\mathcal{K}(T)$. From the previous part of the proof, there exists $M \in \NN$ such that
\[
    \big( \sum_{i=1}^m t^{a_i} \big)^M \big(f - \varepsilon \big(\sum_{i=1}^m t^{a_i}\big)^k\big) = \big( \sum_{i=1}^m t^{a_i} \big)^M f - \varepsilon \big( \sum_{i=1}^m t^{a_i} \big)^{M+k} \in \RR_{\ge 0}[t^{\pm}_1,\ldots,t_n^\pm]_{(k+M)\cdot A}
\]
As the degree $M+k$ polynomial $(\sum_{i=1}^m t^{a_i})^{M+k}$ has only (strictly) positive coefficients in the variables $t^{a_1}, \dots , t^{a_m}$, the above implies that $\big( \sum_{i=1}^m t^{a_i} \big)^M  f\in \RR_{> 0}[t^{\pm}_1,\ldots,t_n^\pm]_{(k+M)\cdot A}$.  This also implies that $\Newt(f) = \conv(k \cdot A)$, since the Newton polytope of a product of polynomials is the Minkowski sum of their Newton polytopes, and concludes the proof of (i) $\Rightarrow$ (ii).

The implication (ii) $\Rightarrow$ (iii) is trivial.

We now show that (iii) $\Rightarrow$ (i). Pick $N \in \mathbb{N}$ such that $g f$ has nonnegative coefficients for $g = (\sum_{i =1}^m t^{a_i})^N$. By~\eqref{Eq:InitialFormsFactor}, for all $w \in \mathbb{R}^n$ we have that $\initial_w ( g f  ) = \initial_w (g) \initial_w (f)$.
Since both $\initial_w ( g f  ) $ and $\initial_w ( g)$ are polynomials with nonnegative coefficients, both are positive on $\mathbb{R}^n_{>0}$. 
Thus $\initial_w(f)(x) > 0$ for all $x \in \mathbb{R}^n_{>0}$ and $w \in \mathbb{R}^n$ as well. Under the assumption $\Newt(f) = \conv(k \cdot A)$, the initial form $\initial_w(f)$ coincides with the truncation of $f$ to the face $G$ of $\conv(k\cdot A)$ with inner normal vector $w$. In particular, $f^{G}(x) > 0$ for all $x \in \mathbb{R}^n_{>0}$. Using Proposition~\ref{Prop:InteriorOfCA_bis}, we conclude that $f$ is strictly $A$-copositive. \end{proof}

\begin{example}
\label{Ex:Running:Sparse}
We return to the polynomial $f$ from Example~\ref{Ex:Running}. In Example~\ref{Ex:Running:Interior}, we showed that $f \in \interior \mathcal{C}_{A}$ for $A = \supp(f)$. By Theorem~\ref{Thm:SparsePolya}, it follows that for $N$ large enough $(\sum_{a \in A}t^a)^Nf$ has positive coefficients. In this example, the smallest $N$ that gives such a certificate of copositivity is $N = 14$, for which the product has 4096 monomials.
\end{example}

\section{Pólya's theorem via Cox coordinates}
\label{sec:Cox}
In this section, we discuss a strategy, alternative to Theorem~\ref{Thm:Main}, for certifying the copositivity of sparse polynomials using Cox coordinates on the toric variety $X_A$. Although this method yields a Pólya-type certificate only when $\conv(A)$ is a product of simplices, in this case it may provide a certificate that is more efficient to compute than the one obtained from Theorem~\ref{Thm:Main}, see Example~\ref{Ex:Running:Cox} below.

\subsection{Cox homogenization}
A polynomial $f$ that is strictly $A$-copositive, might still vanish on the boundary of $\mathbb{R}^n_{\geq0}$, see Examples~\ref{Ex:Running} and~\ref{Ex:Running:Interior}. To determine on which coordinate hyperplanes these zeros occur, it is useful to translate $f$ via a monomial change of coordinates. To that end, we recall a well-known construction from toric geometry: \emph{homogenization in the Cox ring}. For further details on that topic, we refer to \cite[Section 6]{telen2022introductiontoricgeometry} and \cite[Chapter 5]{CoxLittleSchenck}.

For the rest of the section, we fix a finite set $A \subseteq \mathbb{Z}^n$ and assume that $\conv(A)$ is an $n$-dimensional polytope.
We denote by $\Sigma_A$ the inner normal fan of $\conv(A)$, and write $\Sigma_A(d)$ for the set of $d$-dimensional cones in $\Sigma_A$.
We index the rays, that is elements in $\Sigma_A(1)$, by $1,\dots,r$. For each ray, we denote the first lattice point on the ray by $F_i, i=1,\dots, r$.
Furthermore, we call the matrix $F\in \ZZ^{n\times r}$ with columns $F_1, \dots, F_r$ the \emph{ray generator matrix} of $\Sigma_A$. We write $b \in \ZZ^r$ for the vector such that $\conv(A)$ is given by the halfspace description
\begin{align}
\label{Eq:FacetDescr}
    \conv(A) =\{\, a \in \RR^n \,|\, F^\top a+b\geq 0\,\},
\end{align}
where the inequality is read coordinatewise and $F^\top$ is the transpose of $F$. The \emph{Cox homogenization} of $f = \sum_{a \in A}c_at^a$ is given by
\begin{align}
\label{Eq:CoxHom}
    f_\cox(x)\coloneqq x^b f\big(\varphi_{F^\top}(x)\big) = \sum_{a\in A}c_a x^{F^\top a+b}\in \RR[x_1,\dots,x_r],
\end{align}
where $\varphi_{F^\top}$ denotes the monomial map associated with the matrix $F^\top$, cf. \eqref{Eq:MonomialMap}.
Since $f_{\mathrm{cox}}$ has nonnegative exponents, it defines a function on $\mathbb{R}^r_{\geq0}$.
Similar to \eqref{eq:ftildeCD}, we summarize the relevant functions into the diagram 
       \begin{equation}
     \label{eq:CoxDiagram}
                  \begin{tikzcd}
\RR_{>0}^r\arrow[rr,"\varphi_{F^\top}"]\arrow[d,hookrightarrow] & & \RR^n_{>0}\arrow[d,"f"] \\
           \RR^r_{\geq0} \arrow[rr,"f_{\cox}"]  & & \RR
         \end{tikzcd}
     \end{equation}
Although the diagram \eqref{eq:CoxDiagram} does not commute, the copositivity of $f$ and $f_\cox$ are equivalent, since $f_\cox$ and $f \circ \varphi_{F^\top}$ differ only by multiplication with $x^b$.
\begin{lemma}
\label{Lemma:CoxSparseCopisitive}
Let $A \subseteq \mathbb{Z}^n$ be a finite set such that $\conv(A)$ has dimension $n$. For $f \in \RR[t_1^\pm , \dots, t_n^\pm]_A$ denote $f_{\cox}$ its Cox homogenization as in \eqref{Eq:CoxHom}. Then $f$ is $A$-copositive if and only if $f_\cox(x) \geq 0$ for all $x \in \mathbb{R}^r_{\geq0}$. 
\end{lemma}

\begin{proof}
First we show that the monomial map $\varphi_{F^\top}\colon \mathbb{R}^r_{>0} \to \mathbb{R}^n_{>0}$ is surjective. To that end, consider the commutative diagram
     \begin{equation*}
                  \begin{tikzcd}
\RR_{>0}^r\arrow[rr,"\varphi_{F^\top}"]\arrow[d,"\Log"] & & \RR^n_{>0} \\
           \RR^r \arrow[rr,"F"]  & & \RR^n\arrow[u,"\Exp"],
         \end{tikzcd}
     \end{equation*}
where $\Log$ and $\Exp$ denote the coordinate-wise natural logarithm map and exponential map respectively, and $F$ is the linear map given by the matrix $F$. Since $\conv(A)$ is $n$-dimensional, it follows that $F$ has full rank $n$, and therefore the linear map $F\colon \RR^r \to \RR^n$ is surjective. Since $\Log$ and $\Exp$ are bijective, we conclude that $\varphi_{F^\top} = \Exp \, \circ \, F \circ \Log$ is surjective.

With this, it follows that $f_{\cox} = x^b f(\varphi_{F^\top})$ is copositive if and only if $f$ is copositive. Since $f_\cox$ defines a continuous function on $\RR^r_{\geq0}$, the result follows. \end{proof}

\begin{example}
\label{Ex:RunningDehom}
    To illustrate the above constructions, we consider a dehomogenized version of~\eqref{Eq:Running} from Example~\ref{Ex:Running} obtained by setting $t_4 = 1$. 
      \begin{align}
        \label{Eq:RunningDehom}
         f = 1+t_1 - 1.9t_2 + t_2^2  +t_3 + t_1t_3-1.9t_2t_3+t_2^2t_3.
    \end{align}
    Let $A \subseteq \ZZ^3$ denote the support of $f$.
The polytope $\conv(A)$ is full-dimensional and its given by the inequalities $F^\top a + b \geq 0$, where
\begin{align*}
    F^\top=\begin{pmatrix}
        1 & 0 & 0 \\
        0 & 1 & 0 \\
        -2 & -1 & 0 \\
        0 & 0 & 1 \\
        0 & 0 & -1
    \end{pmatrix} \in \mathbb{Z}^{5 \times 3},\qquad  b=\begin{pmatrix}
        0 \\ 0 \\2\\0\\1
    \end{pmatrix}.
\end{align*}
By computing $F^\top a+b$ for each $a \in A$ we get the Cox homogenization of $f$ in~\eqref{Eq:RunningDehom}
\begin{align}
\label{Eq:RunningCoxDehom}
f_\cox = x_3^2x_5 + x_1x_5 -1.9 x_2x_3x_5 + x_2^2x_5+x_3^2x_4 + x_1x_4-1.9x_2x_3x_4+x_2^2x_4.
 \end{align}
 We will return to this polynomial in Example~\ref{Ex:Running:Cox}.
\end{example}

In Proposition \ref{Prop:InteriorOfCA_bis}, we  characterized strictly $A$-copositive polynomials using truncations. In Proposition~\ref{prop:interiorCAcox} below, we give a similar characterization in terms of Cox homogenization.
Using the facet description of $\conv(A)$ in \eqref{Eq:FacetDescr}, the relative interior of every face of $\conv(A)$ has the form
\begin{align}
\label{Eq:JFaces}
 G_I \coloneqq \{ \, a \in \RR^n \, \mid \, F_i \cdot a + b_i = 0 \text{ for all } i \in I \text{ and }  F_i \cdot a + b_i > 0 \text{ for all } i \in [r] \setminus I \, \}
 \end{align}
for some set $I \subseteq [r]=\{1,\ldots,r\}$ of active constraints. Using this notation for the faces of $\conv(A)$, we relate truncations of $f$ to $f_\cox$.

\begin{lemma}
\label{lem:coxtruncation}
Let $A \subseteq \mathbb{Z}^n$ be a finite set such that $\conv(A)$ has dimension $n$.
Let $I \subseteq [r]$ be the set of active constraints for the face $G_I$ of $\conv(A)$ as in \eqref{Eq:JFaces}.
For any $x \in \RR^r_{>0}$ we write $x_I \in \mathbb{R}^r_{\geq0}$ for the vector whose entries $(x_I)_i = 0$ if $i \in I$ and keep  $(x_I)_i = x_i$ if $i \in [r] \setminus I$.
For $f \in \RR[t_1^\pm , \dots, t_n^\pm]_A$ and $x\in \RR^r_{>0}$, we have
\begin{align*}
        f_{\cox}(x_I)=x^bf^{G_I}(\varphi_{F^\top}(x)).
    \end{align*}
\end{lemma}
\begin{proof}
A simple computation shows that
\begin{align}
\label{Eq:ProofCoxFaces}
        x^b f^{G_I}(\varphi_{F^\top}(x))&=\sum_{a\in G_I \cap A}c_a x^{F^\top a+b}.
    \end{align}
    Using the description of $G_I$ in \eqref{Eq:JFaces}, it follows that the polynomial in \eqref{Eq:ProofCoxFaces} does not involve the variables $x_i$ for $i\in I$. For any $a \in A \setminus G_I$ the monomial $x^{F^\top a+b}$ (in $f_\cox$) involves at least one of the variables $x_i, i\in I$. Thus, the polynomial in \eqref{Eq:ProofCoxFaces} equals $f_{\cox}(x_I)$  as desired.
\end{proof}

To describe the zeros of $f_\cox$ in $\mathbb{R}^r_{\geq0}$ in terms of the combinatorics of $A$, we consider the \emph{irrelevant ideal} 
\begin{align*}
    B(A) \coloneqq\left\langle \, \prod_{i\notin \sigma}x_i \,\middle|\,\sigma \in \Sigma_A(n)\,\right\rangle     \subseteq \CC[x_1, \dots , x_r]
\end{align*}
where $i\notin \sigma$ means that the ray spanned by $F_i$ does not belong to the cone $\sigma$. 
The vanishing locus of the irrelevant ideal
\begin{align*}
    Z_A:=\{\,x\in \CC^r\,|\,\forall g\in B(A)\colon g(x)=0\,\}
\end{align*}
is usually a reducible variety. 
We can decompose it into linear irreducible components by considering minimal non-faces.
A subset $C \subseteq [r]$ is called a \emph{primitive collection of rays in $\Sigma_A$} (\cite[Definition 5.1.5]{CoxLittleSchenck}) if 
the indices in $C$ correspond to facet normals that do not form a cone in $\Sigma_A$, but every subset of them does.
We denote by $C_1,\ldots,C_k \subseteq [r]$ the primitive collections of rays in~$\Sigma_A$.
By \cite[Proposition 5.1.6]{CoxLittleSchenck}, $Z_A$ decomposes into irreducible components as
\begin{align}
\label{eq:decompositionIrrelevantLocus}
    Z_A=\bigcup_{j=1}^k\mathcal{V}(x_i\, \mid \, i\in C_j).
\end{align}
By construction, $f_\cox$ vanishes on $Z_A$. For the sake of completeness, we give a proof of this simple fact in the next lemma.
\begin{lemma}
\label{Lemma:VanishingOnIrrelevant}
For $f \in \RR[t_1^\pm , \dots, t_n^\pm]_A$, we have that $f_\cox(x)=0$ for all points in the vanishing locus of the irrelevant ideal $x \in Z_A$. 
\end{lemma}

\begin{proof}
     For any $a\in A$, let $G$ be the smallest face of $\conv(A)$ that contains $a$ and let $v$ be a vertex of $G$ with corresponding maximal cone $\sigma \in \Sigma(n)$.
     For any $i \in [r]$ such that $i\notin \sigma$, the corresponding facet does not contain $v$.
     This facet can also not contain $a$ since otherwise intersecting it with $G$ would give a smaller face containing $a$. Hence $F_i \cdot a+b_{i}>0$. This means that the monomial $x^{F^\top a+b}$ is divisible by $\prod_{i\notin \sigma}x_i$. We therefore get $f_\cox\in B(A)$ and hence $f_\cox(x)=0$ for all $x \in Z_A$. 
\end{proof}

We now characterizes strictly $A$-copositive polynomials in terms of their Cox homogenization.

\begin{proposition}
\label{prop:interiorCAcox}
Let $A \subseteq \mathbb{Z}^n$ be a finite set such that $\conv(A)$ has dimension $n$. 
  A polynomial $f \in \RR[t_1^\pm , \dots, t_n^\pm]_A$ is strictly $A$-copositive if and only if 
  both of the following two conditions hold: \begin{enumerate}
        \item[(i)] $f_\cox$ is copositive and
        \item[(ii)] $\mathcal{V}(f_\cox) \cap \RR^r_{\ge 0} = Z_A \cap \RR^r_{\ge 0}$.
    \end{enumerate}
\end{proposition}

\begin{proof}
By Lemma~\ref{Lemma:CoxSparseCopisitive}, we have that $f$ is $A$-copositive if and only if condition (i) holds. 
Thus, it is enough to prove that an $A$-copositive polynomial $f$ is strictly $A$-copositive if and only if condition (ii) holds. By Proposition \ref{Prop:InteriorOfCA_bis}, we have that $f$ is strictly $A$-copositive if and only if 
\begin{align}
\label{eq:recallTruncationsforCox}
f^{G}(t)>0\text{ for all nonempty faces }G \subseteq \conv(A) \text{ and all }t\in\RR_{>0}^n.
    \end{align}

In the remainder of the proof, we assume that $f$ is $A$-copositive.
To prove that (ii) implies strict $A$-copositivity, let $t \in \RR^n_{>0}$ and let $I \subseteq [r]$ be the set of active constraints for a nonempty face $G_I$ of $\conv(A)$  (cf.~\eqref{Eq:JFaces}).
Since $\varphi_{F^\top}$ maps $\RR_{>0}^r$ surjectively onto $\RR_{>0}^n$ (see the proof of Lemma~\ref{Lemma:CoxSparseCopisitive}), there exists $x\in \RR_{>0}^r$ such that $\varphi_{F^\top}(x)=t$.
By Lemma~\ref{lem:coxtruncation}, we have
\begin{align}
\label{eq:ProofProp4_3}
          f^{G_I}(t)=f^{G_I}(\varphi_{F^\top}(x))=x^{-b}f_{\cox}(x_I),
        \end{align}
        where $x_I$ is given by
        $(x_I)_i = 0$ if $i\in I$  and $(x_I)_i = x_i$ if $i \in [r] \setminus I$.
If $x_I \notin Z_A$, then \eqref{eq:ProofProp4_3}, together with conditions (i) and (ii), implies that \eqref{eq:recallTruncationsforCox} holds, which in turn implies that $f$ is strictly $A$-copositive.

To show that $x_I\notin Z_A$, it suffices to find one polynomial in the irrelevant ideal $B(A)$ which does not vanish at $x_I$. Pick any vertex $v\in G_I$ (here we use that $G_I$ is not the empty face) and let $\sigma \in \Sigma_A(n)$ be the corresponding maximal cone.
By construction, the rays $F_i$ are contained in the cone $\sigma$ for all $i \in I$.
Since $(x_I)_i = x_i > 0$ for $i  \notin \sigma$ by the choice of $x$, we have
\begin{align*}
        \prod_{i\notin \sigma}(x_I)_i=\prod_{i\notin \sigma}x_i>0.
    \end{align*}
    In particular, $x_I\notin Z_A$. This concludes the proof that (ii) implies strict $A$-copositivity.
    
     To prove the other implication, assume that $f$ is strictly $A$-copositive or equivalently that \eqref{eq:recallTruncationsforCox} holds.
     By Lemma~\ref{Lemma:VanishingOnIrrelevant}, we have that $f_\cox(x) = 0$ for all $x \in Z_A$.
     Thus, we only have to show that $f_\cox(z)=0$ for some $z\in \RR_{\geq 0}^r$ implies $z\in Z_A$. 

Let $z \in \RR_{\geq 0}^r$ such that $f_\cox(z) = 0$. Using the notation $I = \{ i \in [r] \mid z_i = 0\} \subseteq [r]$, we associate to $z$ a point $x \in \RR^r_{>0}$, which is defined as $x_i = z_i$ for $i \in [r] \setminus I$ and $x_i = 1$ for $i \in I$. With the notation used in the first part of the proof, we have $z = x_I$.

If the facet normal vectors $\{F_i \mid i \in I\}$ formed a cone in $\Sigma_A$, then this cone would be dual to a face $G_I \subseteq\conv(A)$, and Lemma \ref{lem:coxtruncation} would imply that \begin{align*}
        0=f_\cox(x_I) =x^b f_{|G_I}(\varphi_{F^\top}(x)),
    \end{align*}
    which would contradict \eqref{eq:recallTruncationsforCox}.
    Thus, the normal vectors $\{F_i \mid i \in I\}$ don't span a cone in $\Sigma_A$ and therefore $I$ must contain some primitive collection $C_j \subseteq [r]$. Using \eqref{eq:decompositionIrrelevantLocus}, we have that $x_I \in \mathcal{V}(x_i \mid i\in C_j)\subseteq Z_A$ finishing the proof. 
\end{proof}

\subsection{Preprimes in the Cox ring}

To obtain a Pólya-type certificate for the Cox homogenization (Theorem~\ref{thm:polyaInCox}), we follow the approach of Section~\ref{sec:SparsePolya} and apply the Representation Theorem (Theorem~\ref{Thm:RepThm}). As a first step, we show that a certain preprime is Archimedean. 

\begin{proposition}
\label{Prop:CoxPolyaRepresentation}
Let $A\subseteq \mathbb{Z}^n$ be a finite set such that $\conv(A)$ is full dimensional, and denote $C_1,\dots,C_k \subseteq [r]$ the primitive collections of rays in $\Sigma_A$. If $C_1\cup \cdots \cup C_k=[r]$ then the preprime \begin{align*}
    T\,=\,\RR_{\geq0}[x_1,\dots,x_r]+\left\langle \, 1-\sum_{i\in C_j}x_i \ \middle| \ j\in [k]\,\right\rangle \subseteq \RR[x_1,\dots,x_r]
\end{align*}
is Archimedean. If additionally $f \in \RR[t_1^\pm, \dots, t_n^\pm]_A$ is strictly $A$-copositive, then there exist $g \in \mathbb{R}_{\geq 0}[x_1,\dots,x_r]$ and $g_1,\ldots,g_k\in \RR[x_1,\dots,x_r]$ such that
    \begin{align}
    \label{Eq:CoxPolyaRepresentation}
    f_\cox= g+\sum_{j=1}^k g_j\left(1-\sum_{i\in C_j}x_i\right).
\end{align}
\end{proposition}

\begin{proof}
The proof of the first claim is similar to the proof of Lemma~\ref{Lemma:Archimedian}.
  By Lemma \ref{Lemma:Archimedean}, it suffices to show that $\RR\subseteq H_{T}$ and $x_1,\dots,x_r\in H_T$.
  The inclusion $\RR\subseteq H_{T}$ follows from the fact that $\RR$ is Archimedean.
  To show that $x_i\in H_{T}$, find $j \in [k]$ such that $i \in C_j$.
  This is guaranteed to exist since the primitive collections cover all the rays. We then have $1+x_i\in T$ and $1-x_i=\sum_{\ell\in C_j, \ell \neq i}x_\ell+1-\sum_{\ell\in C_j}x_\ell\in T$. 
  Hence $x_i \in H_{T}$, concluding the proof that $T$ is Archimedean. 
  
To prove the second part of the statement, we consider the set
\begin{align*}
    \mathcal{K}(T) =\{ \, z \in \RR^r \,|\, \forall g\in T\colon g(z)\geq 0\,\}.
\end{align*}
By construction, $\mathcal{K}(T) \subseteq \mathbb{R}^r_{\geq0}$.
Since the polynomials $\pm(1-\sum_{i\in C_j}x_i), \, j=1,\dots,k$ are in $T$, we have that $\sum_{i\in C_j}z_i=1$ for $z \in \mathcal{K}(T)$.
In particular, $z \notin \bigcup_{j=1}^k\mathcal{V}(x_i \,|\, i\in C_j)=Z_A$.
Using this observation, Proposition~\ref{prop:interiorCAcox} implies that $f_\cox(z)>0$ for all $z \in \mathcal{K}(T)$  if $f$ is strictly $A$-copositive.
We now apply the Representation Theorem \ref{Thm:RepThm} to obtain the desired decomposition of $f_\cox$.
\end{proof}

Proposition~\ref{Prop:CoxPolyaRepresentation} decomposes $f_\cox$ as the sum of a polynomial with nonnegative coefficients and a polynomial lying in the ideal generated by the polynomials associated with primitive collections of rays. We wish to eliminate the latter term via a suitable change of variables, as in the proof of Proposition~\ref{prop:polya}. There, a crucial step was using the homogeneity of	$f$ as a function on $Y_{\hat{A}}$. The multi-homogeneity of $f_\cox$ is determined by vectors in the kernel of $F$. For completeness, we recall the proof of this simple fact.

\begin{lemma}
\label{Lemma:fCoxHomog}
    Let $A \subseteq \ZZ^n$ be a finite set and let $F \in \mathbb{Z}^{n \times r}$ be the ray generator matrix of $\Sigma_A$. For every $v \in \ker(F), \lambda \in \mathbb{R}$ and $f\in \RR[t_1^\pm,\dots,t_n^\pm]_A$, we have
    \[ f_\cox(\lambda^{v_1} x_1, \dots, \lambda^{v_r} x_r) = \lambda^{v\cdot b}f_\cox(x), \]
    where $b \in \RR^r$ denotes the vector from \eqref{Eq:FacetDescr}.
\end{lemma}

\begin{proof}
    Using that $(\lambda^{v_1}x_1)^{w_1}\cdots (\lambda^{v_r}x_r)^{w_r}= \lambda^{v^\top  w}x^w =  \lambda^{v \cdot w}x^w$ for any $w\in \mathbb{Z}^r$ and that $v^\top F^\top=0$, a direct computation shows that
    \[
    f_\cox(\lambda^{v_1} x_1, \dots, \lambda^{v_r} x_r) = \sum_{a\in A}c_a \lambda^{v^\top(F^\top a+b)} x^{F^\top a+b} = \lambda^{v\cdot b}f_\cox(x). \hfill\qedhere
    \]
\end{proof}

\subsection{Product of simplices}
\label{sec:ProdSimplex}
To eliminate the terms $g_j\left(1-\sum_{i\in C_j}x_i\right)$ in the certificate \eqref{Eq:CoxPolyaRepresentation}, we now assume that $\conv(A)$ is a product of simplices. Without this additional assumption a Pólya certificate with multiplier $\sum_{i\in C_j}x_i$ might not exist, see the examples below in Section~\ref{sec:Examples}. In the following lemma, we collect basic facts about products of simplices and their normal fans.

\begin{lemma}
\label{lem:simplicesKernel}
\label{Lemma:ProductSimplices}
For $j\in [k]$, let $\Delta_j \subseteq \RR^{n_j}$ be an $n_j$-dimensional simplex, and set $n = n_1 + \dots + n_k$
Let $\Sigma_{\Delta_j}$ be the inner normal fan of $\Delta_j$, and denote by $F^{(j)}$ the ray generator matrix of $\Sigma_{\Delta_j}$.
We enumerate the rays of $\Sigma_{\Delta_1}, \Sigma_{\Delta_2}, \dots, \Sigma_{\Delta_k}$ by $C_1 = \{\,1,\dots,n_1+1\,\}, \, C_2 = \{\, n_1 +2,\dots, n_1+n_2+2 \,\}, \dots ,C_k = \{\, n_1 + n_2 + \dots n_{k-1} + k-1,\dots, n + k \,\}$ respectively. 
We have the following
\begin{itemize}
    \item[(i)]
The ray generator matrix of the normal fan $\Delta = \Delta_1 \times \dots \times \Delta_k$ is given by
 \begin{align*}
F =\begin{pmatrix}
F^{(1)} & 0 & 0 \\
0 & \ddots & 0 \\
0 & 0 & F^{(k)} 
\end{pmatrix} \in \mathbb{Z}^{n  \times (n  + k)}.
\end{align*}
\item[(ii)] The sets $C_1,\ldots,C_k$ are precisely the primitive collections in the normal fan of $\Delta$.
\item[(iii)]
  $  \prod_{j=1}^k\sum_{i\in C_j}x_i=\sum_{\sigma \in \Sigma_\Delta(n)} \prod_{i \notin \sigma} x_i.$
\end{itemize}
\end{lemma}

\begin{proof}
The normal fan of a simplex $\Sigma_{\Delta_j}$ has exactly $n_j + 1$ many rays and every proper subset of them form a cone in $\Sigma_{\Delta_j}$
By \cite[Lemma 7.7]{Ziegler}, the inner normal fan of $\Delta_1 \times \dots \times \Delta_k$  equals
    \begin{align}
    \label{eq:ProductFan}
        \{ \, \sigma_1 \times \dots \times \sigma_k \, \mid \, \sigma_1 \in \Sigma_{\Delta_1}, \dots, \sigma_k \in \Sigma_{\Delta_k} \,  \}.
    \end{align}
Thus, every ray in $\Sigma_\Delta$ has the form $\{0\} \times \dots \times \{0\} \times \sigma_j \times \{0\} \times \dots \times \{0\}$ for some $\sigma_j \in \Sigma_{\Delta_j}(1)$, $j \in [k]$. In particular, $F$ has the desired form.

The above observation also implies that $C_1,\dots,C_k$ are primitive collection of rays. Indeed, for each $j \in [k]$ the collection of rays $\{0\} \times \dots \times \{0\} \times \sigma_i \times \{0\} \times \dots \times \{0\}, \, i \in C_j, \sigma_i \in \Sigma_{\Delta_j}(1)$ does not form a cone in the normal fan $\Sigma_\Delta$ but any proper subset of it does. Thus, $C_j$ is a primitive collection of rays.
  
  On the other hand if $C\subseteq [n+k]$ is such that the corresponding rays do not form a cone in $\Sigma_\Delta$, then there exists an index $j \in [k]$ such that the set $C_j$ of those rays $\sigma_i$ that are in $\Sigma_{\Delta_j}(1)$ do not form a cone in $\Sigma_{\Delta_j}$. Then $C_j \subseteq C$ so that minimality of a primitive collection implies $C_j = C$. So $C_1,\ldots,C_k$ are in fact the only primitive collections.

  By~\eqref{eq:ProductFan}, every maximal cone of $\Sigma_\Delta$ is a product of maximal cones of $\Sigma_{\Delta_1}, \dots \Sigma_{\Delta_k}$. 
  Since these maximal cones are given by all but one rays in $\Sigma_{\Delta_j}$ the monomials on the right-hand side of (iii) correspond to all possible ways of choosing one element from each $C_1,\dots,C_k$. Since these are exactly the monomials on the left-hand side of (iii), the result follows.
\end{proof}

Next, we prove the existence of a Pólya multiplier for the Cox homogenization under the assumption that $\conv(A)$ is a product of \emph{dilated standard simplices}. These are simplices of the form $\Delta = d \cdot \conv(0,e_1,\dots,e_{n})\subseteq \RR^n$ for some $d \in \NN$.
In this case, the ray generator matrix of $\Sigma_{\Delta}$ is given by
\begin{align}
\label{Eq:RayGenDilated}
    F = \begin{pmatrix}
       \mathrm{Id}_{n\times n}    &  -\mathds{1}_{n}    \end{pmatrix} \in \ZZ^{n \times (n +1)},
\end{align}
where $\mathrm{Id}_{n\times n}$ denotes the identity matrix and $\mathds{1}_{n}$ is a vector of length $n$ with all entries equal to~$1$.
With this choice of ordering of the extreme rays, the facet opposite to $e_i$ is given by $F_i$ for $i=1,\dots,n$, while the facet opposite to $0$ has $F_{n+1} =  -\mathds{1}_{n}$ as inner normal vector. 
Moreover, the kernel of $F$ contains the vector $\mathds{1}_{n+1}$.

\begin{theorem}
\label{prop:polyaInCox_standard}
   Let $A \subseteq \ZZ^n$ be a finite set such that $\conv(A) = \Delta_1 \times \dots \times \Delta_k$ is a product of full-dimensional dilated standard simplices, and denote by $C_1,\dots,C_k \subseteq [r]$ the primitive collections of rays in $\Sigma_A$. For $f\in \RR[t_1^{\pm},\ldots,t_n^{\pm}]_A$, the following are equivalent.
\begin{enumerate}
    \item[(i)] $f$ is strictly $A$-copositive.
    \item[(ii)]  $\exists N_1,\ldots,N_k\in \NN\colon\,\prod_{j=1}^k\left(\sum_{i\in C_j}x_i\right)^{N_j }f_\cox \in \RR_{\geq0}[x_1,\dots,x_r]$ and $\Newt(f) = \conv(A)$.
    \item[(iii)] $\exists N\in \NN\colon \, \left(  \sum_{\sigma \in \Sigma_A(n)} \prod_{i \notin \sigma} x_i  \right)^N f_\cox  \in \RR_{\geq0}[x_1,\dots,x_r]$ and $\Newt(f) = \conv(A)$.
\end{enumerate}
\end{theorem}

\begin{proof}
    First, we assume that $f$ is strictly $A$-copositive. By Theorem~\ref{Thm:SparsePolya}, we have $\Newt(f) = \conv(A)$. By Proposition~\ref{Prop:CoxPolyaRepresentation}, there exist $g \in \mathbb{R}_{\geq 0}[x_1,\dots,x_r]$ and $g_1,\ldots,g_k\in \RR[x_1,\dots,x_r]$ with
    \begin{align}
    \label{Eq:CoxPolyaRepresentation_proof}
    f_\cox= g+\sum_{j=1}^k g_j\left(1-\sum_{i\in C_j}x_i\right).
\end{align}
Let $F$ denote the ray generator matrix of $\Sigma_A$. Since $\conv(A)$ is a product of dilated standard simplices, Lemma~\ref{lem:simplicesKernel} implies that $\mathds{1}_{n+k} = (1,\dots,1) \in \ker F$. Thus, $f_\cox$ is multihomogeneous when partitioning the variables according to the primitive collections by Lemma~\ref{Lemma:fCoxHomog}. 
Substituting $x_i\mapsto x' \coloneq \frac{x_i}{\sum_{l\in C_j}x_l}$ where $j \in [k]$ is uniquely determined by $i$ via $i\in C_j$ (see Lemma \ref{lem:simplicesKernel}), we get
\begin{align}
\label{eq:proof_changeofcoords}
    \prod_{j=1}^k\left(\sum_{i\in C_j}x_i\right)^{w_i}f_\cox(x)=g(x')
\end{align}
for some $w_1,\dots,w_k \in \ZZ$.
Now we simply multiply both sides by large enough powers of the linear forms $\sum_{i\in C_j}x_i$ to clear denominators to find $N_1,\dots,N_k$ as in (ii).

The equivalence of (ii) and (iii) follows from Lemma~\ref{Lemma:ProductSimplices}. In the rest of the proof, we show that (iii) implies (i). Let $h = \left(\sum_{\sigma \in \Sigma_A(n)} \prod_{i \notin \sigma} x_i  \right)^N $ such that $hf_\cox \in \RR_{\geq0}[x_1,\dots,x_r]$. Since both $h$ and $hf_\cox$ have nonnegative coefficients, it follows that $f_\cox(x) \geq 0$ for all $x \in \RR^r_{\geq0}$. By Proposition~\ref{prop:interiorCAcox}, to prove strict $A$-copositivity of $f$, it suffices to show that $f_\cox(x) = 0$ if and only if $x \in Z_A$. By Lemma~\ref{Lemma:VanishingOnIrrelevant}, $f_\cox(x) = 0$ for all $x \in Z_A$. 

Let $x^* \in \mathbb{R}^r_{\geq0}$ such that $f_\cox(x^*) = 0$. Since $hf_\cox$ has positive coefficients and $h(x^*)f_\cox(x^*)=0$, every monomial of $hf_\cox$ vanishes at $x^*$. To prove that $x^* \in Z_A = \mathcal{V}( \prod_{i \notin \sigma} x_i \, \mid \, \sigma \in \Sigma_A(n) \,)$, we show that for every $\sigma \in \Sigma_A(n)$ there exists a monomial of $hf_\cox$ that only involves the variables $x_i, i \notin \sigma$. The polynomial $h$ certainly has such a monomial, we now prove that $f_\cox$ does so as well.
Let $a \in A$ be the vertex of $\conv(A)$ corresponding to the maximal cone $\sigma$. 
Since $\conv(A)$ is the Newton polytope of $f$, the monomial $t^a$ appears in $f$.
The corresponding term in $f_\cox$ is $x^{F^\top a+b}$.
Since $\sigma$ contains precisely those facets which pass through the vertex $a$, we have $(F^\top a+b)_i=0 \Leftrightarrow i \in \sigma$ and hence $x^{F^\top a+b}$ contains only variables $x_i$ with $i\notin \sigma$.

Define $w\in \RR_{\geq 0}^n$ by \begin{align*}
    w_i=\begin{cases}
        1 & \text{ if } i \notin \sigma \\ 0 & \text{ else }
    \end{cases}
\end{align*}
then we have $\initial_w(hf_\cox)=\initial_w(h)\initial_w(f_\cox)\neq 0$. Since both $h$ and $f_\cox$ are homogeneous and contain a monomial only in the variables $x_i$ for $i\notin \sigma$, the initial forms of $h$ and of $f_\cox$ only use these variables. It follows that $\initial_w(hf_\cox)$ only uses the variables $x_i$ for $i\notin \sigma$, in particular $hf_\cox$ contains such a monomial.
\end{proof}
\begin{remark}
    By a refined analysis as in the proof of Theorem \ref{Thm:SparsePolya}, and by possibly increasing the exponents $N_1,\ldots, N_k$ and $N$ in Theorem~\ref{prop:polyaInCox_standard}, we can additionally always ensure that \begin{align*}
        \supp\left(\prod_{j=1}^k\left(\sum_{i\in C_j}x_i\right)^{N_j }f_\cox\right)&=\sum_{j=1}^kN_j\cdot \supp\left(\sum_{i\in C_j}x_i\right)+\supp(f_\cox) \\
        \supp\left(\left(  \sum_{\sigma \in \Sigma_A(n)} \prod_{i \notin \sigma} x_i  \right)^N f_\cox \right)&=N\cdot \supp\left( \sum_{\sigma \in \Sigma_A(n)} \prod_{i \notin \sigma} x_i\right)+\supp(f_\cox).
    \end{align*} 
\end{remark}

In the remainder of this section, we extend Theorem~\ref{prop:polyaInCox_standard} to arbitrary simplices. The key step is to reduce to the case of standard simplices via an affine transformation of the exponent vectors.
To fix the notation, we give an explicit construction of such an affine transformation.
For an $n$-dimensional simplex $\Delta \subseteq \mathbb{R}^n$ with vertices $a_1,\dots,a_{n+1} \in \ZZ^n$, we consider the matrix
\[M = \begin{pmatrix}
    a_1 - a_{n+1} & \dots &a_n - a_{n+1}
\end{pmatrix} \in \ZZ^{n \times n}.\]
Since $\Delta$ has dimension $n$, the matrix $M$ is invertible over $\mathbb{Q}$. 
The affine transformation 
\begin{align}
\label{Eq:affineCoordchange}
   \psi\colon \RR^n \to \RR^n, \quad  p \mapsto \det(M)M^{-1}\left(p - a_{n+1} \right)
\end{align}
is defined over $\mathbb{Z}$ and maps the vertices $a_1,\dots,a_n,a_{n+1}$ to the vectors $\det(M)e_1,\dots,\det(M)e_n,0$. In particular, the image of $\Delta$ under \eqref{Eq:affineCoordchange} is a dilated standard simplex.
The multiplication by $\det(M)$ ensures that the affine transformation is defined over $\mathbb{Z}$ and therefore maps all lattice points in $\Delta$ to lattice points. Let $F$ denote the ray generator matrix of the inner normal fan $\Sigma_\Delta$ of $\Delta$. We choose the ordering of the rays such that the $i$-th column of $F$ is an inner normal vector of the facet opposite to the vertex $a_i$ for $i=1,\dots,n+1$.
Let $b \in \ZZ^r$ such that $\Delta = \{ \, p \in \RR^n \, | \, F^\top p + b \geq 0 \, \}$.
   A simple computation shows that 
    \[  \psi(\Delta)  =   \{ \, q \in \RR^n \, | \, F^\top \psi^{-1}(q) + b \geq 0 \, \} = \{ \, q \in \RR^n \, | \, \tfrac{1}{\det(M)} F^\top M q + (b + F^\top a_{n+1}) \geq 0 \, \}.\]
The columns of the matrix $\tfrac{1}{\det(M)} M^\top F$ generate the rays in the normal fan of $\psi(\Delta)$.
However, these columns might not have integer entries. To get the primitive ray generators of $\Sigma_{\psi(\Delta)}$, we multiply each column by $v_1,\dots,v_{n+1} \in \RR$. These are uniquely given by
\begin{align}
\label{eq:Defv_Ftilde}
   \tfrac{1}{\det(M)} M^\top F \diag(v_1,\dots,v_{n+1}) = \begin{pmatrix}
        \mathrm{Id}_{n\times n} & -\mathds{1}_n    \end{pmatrix}.
\end{align}
or equivalently
\begin{align}
\label{eq:Defv}
   v_i = \tfrac{\det(M)}{F_i \cdot a_i + b_i}, \qquad \text{for } i =1,\dots,n+1.
\end{align}
Note that the matrix in the left-hand side of \eqref{eq:Defv_Ftilde} is the ray generator matrix of the dilated simplex $\psi(\Delta)$, cf. \eqref{Eq:RayGenDilated}.

\begin{lemma}
\label{lemma:v_inKernel}
    For $v = (v_1,\dots,v_{n+1}) \in \RR^{n+1}$ as defined in \eqref{eq:Defv}, we have that $v \in \ker (F) \cap {\ZZ^{n+1}}$ and all $v_1,\dots,v_{n+1}$ have the same sign.
\end{lemma}

\begin{proof}
We use the notation from above.
Since by definition $F_i \cdot a_i + b_i \geq 0$, the values of $v_i$'s have the same sign.
By \cite[Proposition 1]{Bruns} $\tfrac{\det(M)}{F_i \cdot a_i + b_i}$ equals the lattice volume of the facet opposite to $a_i$, which implies that $v \in \ZZ^{n+1}$. 
Since the all-one vector $\mathds{1}_{n+1}$ lies in the kernel of the matrix on the right-hand side of \eqref{eq:Defv_Ftilde}, we have 
    \[ 0 =  \tfrac{1}{\det(M)} M^\top F \diag(v_1,\dots,v_{n+1}) \mathds{1} =  \tfrac{1}{\det(M)} M^\top F  v. \]
Since $\tfrac{1}{\det(M)} M^\top$ is invertible, it follows that $v \in \ker(F)$.
\end{proof}

If $\Delta = \Delta_{n_1} \times \dots \times \Delta_{n_k} \subseteq \RR^n$, we apply the above construction componentwise and obtain an affine linear map $\Phi\colon \RR^n \to \RR^n, \;  p \mapsto L p  + w$, where 
\begin{align}
\label{eq:Landw}
        L = \begin{pmatrix}
\det(M_1) M_1^{-1} & 0 & 0 \\
0 & \ddots & 0 \\
0 & 0 & \det(M_k) M_k^{-1}  
\end{pmatrix}, \qquad w = \begin{pmatrix}
-\det(M_1) M_1^{-1} a_{n_1+1}^{(1)}  \\
 \ddots \\
 -\det(M_k) M_k^{-1} a_{n_k+1}^{(k)}  
\end{pmatrix}
\end{align}
and $a^{(j)}_{n_j+1}$ denotes an arbitrary but fixed vertex of $\Delta_{n_j}$ for each $j \in [k]$.
Using Lemma~\ref{lem:simplicesKernel}(i) and \eqref{eq:Defv}, we have that the ray generator matrix of $L \Delta + w$ is given by 
\begin{align}
    \label{eq:ChoiceOfv}
    \widetilde{F} = (L^{-1})^\top F \diag(v)
\end{align}
for a unique choice of $v \in \ZZ^n$, where $F$ denotes the ray generator matrix of $\Sigma_\Delta$. Since $\Phi$ maps lattice points to lattice points, for any $f = \sum_{a\in A}c_a x^a \in \RR[t_1^\pm,\dots,t_n^\pm]$ the function 
\[\widetilde{f}(t) \coloneqq   \sum_{a\in A} c_a t^{La + w}\]
is a Laurent polynomial. In the next lemma, we relate the Cox homogenization of $f$ and $\widetilde{f}$.

\begin{lemma}
\label{Lemma:CoordChange}
Let $A \subseteq \ZZ^n$ be a finite set such $\conv(A)$ is a product of (not necessary standard) full-dimensional simplices $\Delta_1, \dots , \Delta_k$, and let $L \in \mathbb{R}^{n\times n},\, w \in \RR^n,\, v \in \ZZ^{n+k}$ be as defined in \eqref{eq:Landw} and \eqref{eq:ChoiceOfv}.
For all $f \in \RR[t_1^\pm, \dots, t_n^\pm]_A$ the Cox homogenization of $\widetilde{f}$ satisfies $f_\cox(x^v) = \widetilde{f}_\cox(x)$.
\end{lemma}

\begin{proof}
   Let $V = \diag(v)$ and let $F \in \ZZ^{n \times r}$ be the ray generator matrix of $\Sigma_\Delta$ and let $b \in \ZZ^r$ such that $\Delta = \{ \, p \in \RR^n \, | \, F^\top p + b \geq 0 \, \}$.
    By construction, we have
    \[L \Delta + w  = \{ \, q \in \RR^n \, | \, (V F^\top L^{-1}) q + V(b - F^\top L^{-1}w) \geq 0 \, \}.\]
    Moreover, the rays of the normal fan of $L\Delta + w$ are spanned by the rows of the matrix $V F^\top L^{-1}$. 
    Thus, for $f \in \RR[t_1,\dots,t_n]_A$ the Cox homogenixation of $\widetilde{f}$ is given by
    \begin{equation*}
        \widetilde{f}_\cox(x) = \sum_{a \in A} c_a x^{V F^\top L^{-1}(La+w) + V(b - F^\top L^{-1}w)} = \sum_{a \in A} c_a x^{V F^\top a  + V b} = f_\cox(t^v).  \hfill\qedhere
    \end{equation*}
\end{proof}
\bigskip

We finally state the main theorem of this section.

\begin{theorem}
\label{thm:polyaInCox}
Let $A \subseteq \ZZ^n$ be a finite set such $\conv(A)$ is a product of full-dimensional simplices $\Delta_1, \dots , \Delta_k$. Denote by $C_1,\dots,C_k \subseteq [r]$ the primitive collections of rays in $\Sigma_A$, and let $v \in \mathbb{Z}^r$ be defined as in~\eqref{eq:ChoiceOfv}. For $f\in \RR[t_1^{\pm},\ldots,t_n^{\pm}]_A$, the following are equivalent
\begin{enumerate}
    \item[(i)] $f$ is strictly $A$-copositive.
    \item[(ii)]  $\exists N_1,\ldots,N_k\in \NN\colon\,\prod_{j=1}^k\left(\sum_{i\in C_j} x_i \right)^{N_j }f_\cox(x^v) \in \RR_{\geq0}[x_1,\dots,x_r]$  and $\Newt(f) = \conv(A)$.
    \item[(iii)] $\exists N\in \NN\colon \, \left(  \sum_{\sigma \in \Sigma_A(n)} \prod_{i \notin \sigma} x_i  \right)^N f_\cox(x^v)$ has positive coefficients and $\Newt(f) = \conv(A)$.
\end{enumerate}
\end{theorem}

\begin{proof}
    The polynomial $f$ is strictly $A$-copositive if and only if $\widetilde{f}$ is strictly $(LA+w)$-copositive. Now the statement follows from Lemma~\ref{Lemma:CoordChange} and Theorem~\ref{prop:polyaInCox_standard}.
\end{proof}

\section{Examples and counterexamples}
\label{sec:Examples}
In this section, we discuss the Pólya certificates for sparse polynomials from Section~\ref{sec:SparsePolya} and~\ref{sec:Cox} and consider their possible generalizations through a series of examples and counterexamples. We begin with an example that illustrates the objects studied in Section~\ref{sec:ProdSimplex} and compares the Pólya multipliers given in Theorem~\ref{Thm:SparsePolya} and Theorem~\ref{thm:polyaInCox}.
\begin{example}
\label{Ex:Running:Cox}
In the following, we follow the notations of~\eqref{eq:Landw}.
Consider the support set $A$ from Example~\ref{Ex:RunningDehom}. The convex hull of $A$ is the product of the simplices $\Delta_1 = \conv((0,0),(1,0),(0,2)) \subseteq \RR^2$ and $\Delta_1 = \conv((0,1)) \subseteq \RR$. To transform these to standard simplices, we consider $a_2^{(1)} = (0,0), a_1^{(2)} = 0$ and 
\begin{align*}
    M_1 = \begin{pmatrix}
        1 & 0 \\
        0 & 2
    \end{pmatrix} \in \ZZ^{2 \times 2}, \qquad     M_2 = \begin{pmatrix}
        1
    \end{pmatrix} \in \ZZ^{1 \times 1}.
\end{align*}
These give rise to the affine linear map $\Phi\colon \RR^3 \to \RR^3, p \mapsto Lp + w$ where
\begin{align*}
    L = \begin{pmatrix}
        2 & 0 & 0\\
        0 & 1 & 0\\
        0 & 0 & 1
    \end{pmatrix}, \qquad w = \begin{pmatrix}
        0\\
        0\\
        0
    \end{pmatrix}.
\end{align*}
The transformed polytope $\Phi(\Delta_1 \times \Delta_2)$ is a $2$-dilated standard simplex in $\RR^3$. For $f$ as in \eqref{Eq:RunningDehom}, the transformed polynomial reads
 \begin{align*}
         \widetilde{f} = 1+t_1^2 - 1.9t_2 + t_2^2  +t_3 + t_1^2t_3-1.9t_2t_3+t_2^2t_3 = f(t_1^2,t_2,t_3).
    \end{align*}
and its Cox homogenization equals
\begin{align*}
    \widetilde{f}_\cox = x_3^2x_5 + x_1^2 x_5 -1.9 x_2x_3x_5 + x_2^2x_5+x_3^2x_4 + x_1^2x_4-1.9x_2x_3x_4+x_2^2x_4.
\end{align*}
To compare $\widetilde{f}_\cox$ with $f_\cox$ from~\eqref{Eq:RunningCoxDehom}, we compute the vector $v$ as in~\eqref{eq:ChoiceOfv}. This is the unique vector $v$ satisfying 
\begin{align}
    (L^{-1})^\top F \diag(v) = \begin{pmatrix}
        \tfrac{1}{2}v_1 & 0 & -v_3 & 0 & 0 \\
        0 & v_2 & -v_3 & 0 & 0 \\
        0 & 0 &  0 & v_4 & -v_5 \\
    \end{pmatrix} = \begin{pmatrix}
        1 & 0 & -1 & 0 & 0 \\
        0 & 1 & -1 & 0 & 0 \\
        0 & 0 &  0 & 1 & -1 \\
    \end{pmatrix}
\end{align}
 Thus $v = (2,1,1,1,1)$ and we have $\widetilde{f}_\cox = f_\cox(x_1^2,x_2,x_3,x_4,x_5)$ by Lemma~\ref{Lemma:CoordChange}.

The irrelevant ideal has 6 monomial generators corresponding to the 6 maximal cones in $\Sigma_A$, namely 
\begin{align*}
     B(A) =\langle\, x_3x_5,\, x_2x_5,\, x_1x_5,\, x_3x_4,\, x_2x_4,\, x_1x_4\,\rangle=\langle\, x_1,\,x_2,\,x_3\,\rangle \cap\langle\, x_4,\,x_5\,\rangle.
 \end{align*}
 There are two primitive collections of rays, namely $C_1=\{1,2,3\}$ and $\,C_2=\{4,5\}$. By Theorem~\ref{thm:polyaInCox}(iii), $f$ is strictly $A$-copositive if and only if there exists $N\in \NN$ such that \begin{align}
 \label{eq:examplewithirrelevantgens}
(x_3x_5+x_2x_5+x_1x_5+x_3x_4+x_2x_4+x_1x_4)^N f_\cox(x_1^2,x_2,x_3,x_4,x_5)
 \end{align}
 has positive coefficients. In this example the smallest exponent with that property is $N=38$, for which the product in \eqref{eq:examplewithirrelevantgens} has 34320 terms. Part (ii) of Theorem \ref{thm:polyaInCox} states that the strict $A$-copositivity of $f$ can also be certified by computing
 \begin{align}
 \label{eq:examplewithprimitivegens}
     (x_1+x_2+x_3)^{N_1}(x_4+x_5)^{N_2}f_\cox(x_1^2,x_2,x_3,x_4,x_5).
 \end{align}
 For $N_1=38,N_2=0$ the polynomial in (\ref{eq:examplewithprimitivegens}) has only 1716 terms each with a positive coefficient, thereby certifying copositivity. Alternatively, we also compute the certificate from Theorem \ref{Thm:SparsePolya}: 
 \[ (1+t_1+t_2+t_3+t_2^2+t_1t_3+t_2t_3+t_2^2t_3)^N f\]
 has nonnegative coefficients for $N\geq 14$. In this case, the product has 4096 terms.
\end{example}

We now provide two examples showcasing that the special assumptions on the polytope $\conv(A)$ cannot easily be weakened in Theorem \ref{thm:polyaInCox}.
The multipliers involving $f_\cox(x^v)$ from (ii) and (iii) in Theorem \ref{thm:polyaInCox} can be defined for any support set $A$.
  By Lemma~\ref{lemma:v_inKernel}, the vector $v \in \ZZ^r$ lies in the kernel of the ray generator matrix $F$, and without loss of generality we might assume that it has only positive coordinates. 
  In the following two examples, we illustrates that (ii) and (iii) in Theorem \ref{thm:polyaInCox} might fail to detect copositivity for any choice of $v \in \ker(F) \cap \ZZ^{r}_{>0} $.
    
\begin{example}
\label{ex:FirstCounterExample}
    To showcase that the multipliers in Theorem \ref{thm:polyaInCox}(ii) and (iii) cannot be used to certify copositivity, we consider the polynomial \begin{align*}
        f=1+t_1+t_2+t_1^2-2t_1t_2+t_2^2+t_1t_2^2+t_1^2t_2.
    \end{align*}
   By Theorem \ref{Thm:Main}, $f$ is strictly $A$-copositive for $A = \supp(f)$ since the polynomial\begin{align*}
(1+t_1+t_2+t_1^2+t_1t_2+t_2^2+t_1t_2^2+t_1^2t_2)f 
    \end{align*}
    has positive coefficients. We depicted the polytope $\conv(A)$ in Figure~\ref{FIG2}(a). Notice in particular that $\conv(A)$ does not satisfy the assumptions in Theorem~\ref{thm:polyaInCox}. It has $5$ facets, and its ray generator matrix is given by 
\begin{align*}
F = \begin{pmatrix}
    1 & 0 & -1 & -1 & 0 \\
    0 & 1 &  0 & -1 & -1
\end{pmatrix},
\end{align*}
which gives rise to the Cox homogenization
 \begin{align*}
f_\cox=  x_3^2x_4^3x_5^2 + x_1x_3x_4^2x_5^2 + x_2x_3^2x_4^2x_5 + x_1^2x_4x_5^2 -2x_1x_2x_3x_4x_5+x_2^2x_3^2x_4+x_1x_2^2x_3+x_1^2x_2x_5
    \end{align*}

    The matrix $F$ has a three-dimensional kernel,  and each vector $v \in \ker(F) \cap \ZZ^5_{>0}$ has the form $v=( \lambda_1 + \lambda_2, \lambda_2 + \lambda_3, \lambda_1,\lambda_2,\lambda_3)$ for some $\lambda_1,\lambda_2,\lambda_3 \in \ZZ_{>0}$.  
 The Cox homogenization is the homogeneous polynomial
  \begin{align*}
f_\cox&(x^v) =  x_3^{2\lambda_1}x_4^{3\lambda_2}x_5^{2\lambda_3} + x_1^{\lambda_1 + \lambda_2} x_3^{\lambda_1} x_4^{2\lambda_2} x_5^{2\lambda_3} + x_2^{\lambda_2 + \lambda_3} x_3^{2\lambda_1} x_4^{2\lambda_2}x_5^{\lambda_3} + x_1^{2(\lambda_1+\lambda_2)}x_4^{\lambda_2}x_5^{2\lambda_3} \\
&-2x_1^{\lambda_1 + \lambda_2} x_2^{\lambda_2 + \lambda_3} x_3^{\lambda_1} x_4^{\lambda_2} x_5^{\lambda_3}+x_2^{2(\lambda_2 + \lambda_3)}x_3^{2\lambda_1}x_4^{\lambda_2}+x_1^{\lambda_1+\lambda_2} x_2^{2(\lambda_2 + \lambda_3)} x_3^{\lambda_1}+x_1^{2(\lambda_1 + \lambda_2)}x_2^{\lambda_2 + \lambda_3}x_5^{\lambda_3}
    \end{align*}
    of total degree $2\lambda_1 + 3 \lambda_2 + 2 \lambda_3$.
    The irrelevant ideal is generated by the five monomials $x_1x_2x_3,\,x_2x_3x_4,\,x_3x_4x_5,\,x_4x_5x_1,\,x_5x_1x_2$.
In contrast to Theorem~\ref{thm:polyaInCox}, which applies when $\conv(A)$ is a product of simplices, the copositivity of $f$ cannot be certified by computing
  \begin{align}
  \label{eq:counterex1}
(x_1 x_2 x_3\,+\,x_2 x_3 x_4\,+\,x_3 x_4 x_5\,+\,x_4 x_5 x_1\,+\,x_1 x_2)^N f_{\cox}(x^v),
    \end{align} 
    since the monomial $x_1^{N+(\lambda_1+\lambda_2)}x_2^{N+(\lambda_2 + \lambda_3)}x_3^{N+\lambda_1}x_4^{\lambda_2}x_5^{\lambda_3}$ has a negative coefficient $-2$ for any $N \in \mathbb{N}$. 
To justify this claim, we argue that there is unique way to obtain this monomial in the product~\eqref{eq:counterex1}, namely as the product $x_1^{N}x_2^{N}x_3^{N}$ and $-2x_1^{\lambda_1 + \lambda_2}x_2^{\lambda_2 + \lambda_3}x_3^{\lambda_1}x_4^{\lambda_2}x_5^{\lambda_3}$. We first argue that the only monomial of $f_\cox(x^v)$ that can contribute to the coefficient of $x_1^{N+(\lambda_1+\lambda_2)}x_2^{N+(\lambda_2 + \lambda_3)}x_3^{N+\lambda_1}x_4^{\lambda_2}x_5^{\lambda_3}$ is $-2x_1^{\lambda_1 + \lambda_2}x_2^{\lambda_2 + \lambda_3}x_3^{\lambda_1}x_4^{\lambda_2}x_5^{\lambda_3}$. By looking at the degree in the variable $x_1$ we see that a contributing monomial must be divisible by $x_1^{\lambda_1+\lambda_2}$, similarly by looking at the degrees in $x_2$ and $x_3$ we get divisibility by $x_2^{\lambda_2+\lambda_3}$ and $x_3^{\lambda_1}$. This rules out all but two monomials in $f_\cox(x^v)$, leaving us with $-2x_1^{\lambda_1 + \lambda_2}x_2^{\lambda_2 + \lambda_3}x_3^{\lambda_1}x_4^{\lambda_2}x_5^{\lambda_3}$ and $x_1^{\lambda_1+\lambda_2}x_2^{2(\lambda_2+\lambda_3)}x_3^{\lambda_1}$. The second monomial does not work since in order to complement it to the desired monomial $x_1^{N+(\lambda_1+\lambda_2)}x_2^{N+(\lambda_2 + \lambda_3)}x_3^{N+\lambda_1}x_4^{\lambda_2}x_5^{\lambda_3}$ by taking its product with $N$ monomials in the multiplier, we would be forced to pick only monomials in the multiplier that are divisible by $x_1x_3$ (this in order to make the degree in $x_1$ and $x_3$ match). The only such monomial is $x_1x_2x_3$, which does not contain $x_4,x_5$. Hence these terms cannot contribute to the coefficient of $x_1^{N+(\lambda_1+\lambda_2)}x_2^{N+(\lambda_2 + \lambda_3)}x_3^{N+\lambda_1}x_4^{\lambda_2}x_5^{\lambda_3}$ which is divisible by $x_4x_5$. For the first monomial  $-2x_1^{\lambda_1 + \lambda_2}x_2^{\lambda_2 + \lambda_3}x_3^{\lambda_1}x_4^{\lambda_2}x_5^{\lambda_3}$ in $f_\cox(x^v)$, there is a unique way to complement it to desired monomial $x_1^{N+(\lambda_1+\lambda_2)}x_2^{N+(\lambda_2 + \lambda_3)}x_3^{N+\lambda_1}x_4^{\lambda_2}x_5^{\lambda_3}$ by picking $N$ times the monomial $x_1x_2x_3$ in the multiplier. Hence in the product we get the negative coefficient $-2$, and thus Theorem~\ref{thm:polyaInCox}(iii) does not apply in this example.
    
    The primitive collections of rays in $\Sigma_A$ are $\{1,3\},\{1,4\},\{2,4\},\{2,5\},\{3,5\}$. Since these are not disjoint, the multiplier used in~\eqref{eq:counterex1} is not equal to the product of sums over primitive collections. 
One might therefore hope that the multiplier from part (ii) of Theorem~\ref{thm:polyaInCox} could be used to certify copositivity of $f$. This, however, is not the case, since
\begin{align}
  \label{eq:counterex2}
        (x_1+x_3)^{N_1}(x_1+x_4)^{N_2}(x_2+x_4)^{N_3}(x_2+x_5)^{N_4}(x_3+x_5)^{N_5}f_\cox(x^v)
    \end{align}
    also has negative coefficients for any choice of $N_1,N_2,N_3,N_4,N_5$.
   To see this notice first that by possibly increasing some of the $N_i$, we can assume $N_1=N_2=N_3=N_4=N_5\eqqcolon N$.  Then a similar argument as above shows that the polynomial in \eqref{eq:counterex2} always contains the term $-2x_1^{N+\lambda_1 + \lambda_2}x_2^{N+\lambda_2 + \lambda_3}x_3^{2N+\lambda_1}x_4^{N+\lambda_2}x_5^{\lambda_3}$, which has a negative coefficient. Hence criterion (ii) in Theorem~\ref{thm:polyaInCox} also fails to detect copositivity of $f$. 
\end{example}

\begin{figure}[t]
\centering
\begin{minipage}[h]{0.3\textwidth}
\centering
\includegraphics[scale=0.4]{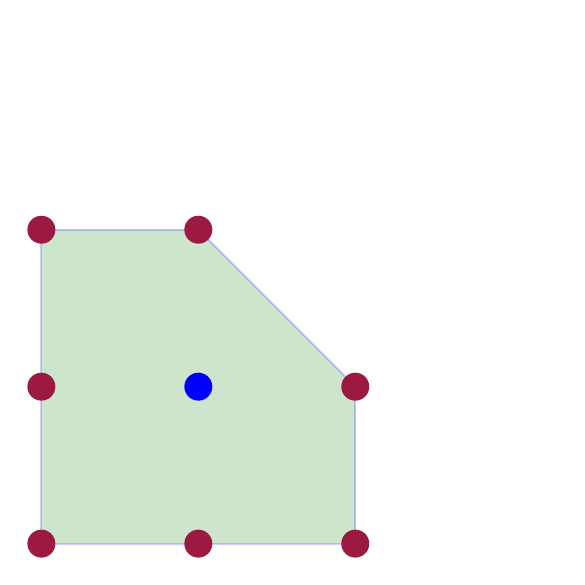}

{\small (a)}
\end{minipage}
\hspace{1 pt}
\begin{minipage}[h]{0.3\textwidth}
\centering
\includegraphics[scale=0.4]{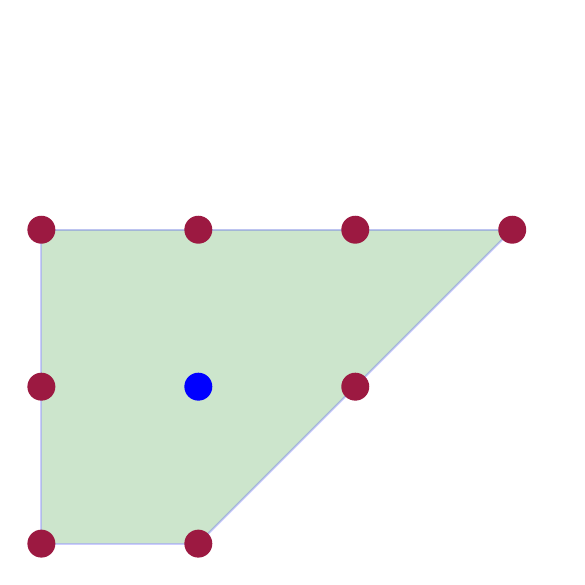}

{\small (b)}
\end{minipage}
\hspace{1 pt}
\begin{minipage}[h]{0.3\textwidth}
\centering
\includegraphics[scale=0.4]{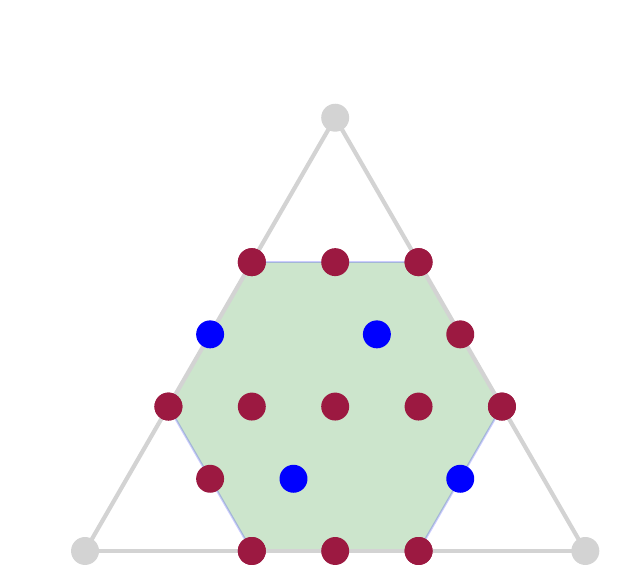}

{\small (c)}
\end{minipage}
\caption{{\small  Newton polytopes of the polynomials considered in Example~\ref{ex:FirstCounterExample},~\ref{ex:counterex2} and~\ref{Ex:MinkowskiSum}. }}\label{FIG2}
\end{figure}

One might be tempted to believe that the criterion in Theorem~\eqref{thm:polyaInCox}(ii) fails in Example~\refeq{ex:FirstCounterExample} because the primitive collections are not disjoint. Recall that disjointness played a crucial role in the proof of Theorem~\ref{prop:polyaInCox_standard}, as it enabled a convenient change of variables in \eqref{eq:proof_changeofcoords} that eliminated all terms except those with positive coefficients.
However, even when the primitive collections are disjoint, we get another problem: the positive kernel vector $v$ in general only makes $f_\cox(x^v)$ homogeneous with respect to the total degree of all variables. 
But in the proof of Theorem \ref{thm:polyaInCox} we also need homogeneity with respect to the total degree of variables within the same primitive collection.
For a product of simplices, this is guaranteed by the block diagonal shape of the ray generator matrix (Lemma~\ref{lem:simplicesKernel}), but the conclusion fails in general as the following examples illustrate.

\begin{example}
\label{ex:counterex2}
    Consider the polynomial \begin{align*}
f= 1 + t_1 + t_2 -2t_1t_2  +t_1^2t_2  +t_2^2  +t_1t_2^2 +t_1^2t_2^2 + t_1^3t_2^2
    \end{align*}
 It is strictly $A$-copositive for $A = \supp(f)$ since \begin{align*}
(1 + t_1 + t_2 + t_1t_2  +t_1^2t_2  +t_2^2  +t_1t_2^2 +t_1^2t_2^2 + t_1^3t_2^2)f 
    \end{align*}
 has nonnegative coefficients, so we can apply Theorem \ref{Thm:Main}. The Newton polytope of $f$ is a quadrilateral, see Figure~\ref{FIG2}(b). 
The ray generator matrix of the normal cone of $\conv(A)$ equals
\begin{align}
    F = \begin{pmatrix}
        1 & 0 & -1 & 0 \\
        0 & 1 &  1 & -1
    \end{pmatrix}.
\end{align}
The Cox homogenization of $f$ is given by
\begin{align*}
f_\cox= x_3 x_4^2 + x_1x_4^2 + x_2x_3^2 x_4  - 2 x_1x_2x_3x_4  + x_1^2 x_2x_4 + x_2^2 x_3^3 + x_1 x_2^2 x_3^2 + x_1^2 x_2^2 x_3 +x_1^3 x_2^2
\end{align*}
In this example, the vectors in $\ker(F)\cap \ZZ^4_{>0}$ are given by $v=(\lambda_1, \lambda_1 + \lambda_2 ,\lambda_1 ,2\lambda_1 + \lambda_2)$ for $\lambda_1,\lambda_2 \in \ZZ_{>0}$, which give 
\begin{equation}
\small
\begin{aligned}
\label{eq:examplefcoxno-2trick}
&f_\cox(x^v)= x_3^{\lambda_1} x_4^{2(2\lambda_1+\lambda_2)} + x_1^{\lambda_1}x_4^{2(2\lambda_1 + \lambda_2)} + x_2^{\lambda_1 + \lambda_2} x_3^{2\lambda_1 } x_4^{2\lambda_1 + \lambda_2}  - 2 x_1^{\lambda_1}x_2^{\lambda_1 + \lambda_2}x_3^{\lambda_1}x_4^{2\lambda_1 + \lambda_2} \\ & + x_1^{2\lambda_1} x_2^{\lambda_1+\lambda_2} x_4^{2\lambda_1 + \lambda_2} + x_2^{2(\lambda_1 + \lambda_2)} x_3^{3\lambda_1} + x_1^{\lambda_1} x_2^{2(\lambda_1 + \lambda_2)} x_3^{2\lambda_1} + x_1^{2\lambda_1} x_2^{2(\lambda_1+\lambda_2)} x_3^{\lambda_1} +x_1^{3\lambda_1} x_2^{2(\lambda_1 + \lambda_2)}.
\end{aligned}
\end{equation}
The primitive collections are $\{1,3\},\{2,4\}$.
The polynomial $f_\cox(x^v)$ is homogeneous with respect to the weight $(1,1,1,1)$, but not with respect to $(1,0,1,0)$ and $(0,1,0,1)$ individually.
Both criteria (ii) and (iii) of Theorem \ref{thm:polyaInCox} fail to detect the copositivity of $f$. 
To see this, note that $(x_1+x_3)(x_2+x_4)=x_1x_2+x_1x_4+x_2x_3+x_3x_4$. Therefore, it is enough to show that condition (iii) fails.
We claim that for any $N\geq \lambda_1$ the polynomial
\begin{align*}
(x_1x_2+x_1x_4+x_2x_3+x_3x_4)^{2N}f_\cox(x^v)
\end{align*}
 has negative coefficient $2\left(\binom{2N}{N-\lambda_1}-\binom{2N}{N}\right)<0$ in front of $x_1^{N+\lambda_1}x_2^{2N+\lambda_1+\lambda_2}x_3^{N+\lambda_1}x_4^{2\lambda_1+\lambda_2}$. Once this claim is proven, since no large even power $N$ exists for which the product in (iii) of Theorem~\ref{thm:polyaInCox} has positive coefficients, no $N'>0$ can work.
To justify the claim, we count all possible ways to obtain the monomial $x_1^{N+\lambda_1}x_2^{2N+\lambda_1+\lambda_2}x_3^{N+\lambda_1}x_4^{2\lambda_1+\lambda_2}$ by multiplying $2N$ monomials from the multiplier $x_1x_2+x_1x_4+x_2x_3+x_3x_4$ and one monomial from $f_\cox(x^v)$. We first look at the possible monomials in $f_\cox(x^v)$ that can be used. By looking at the degree in $x_2$ we see that such a monomial must be divisible by $x_2^{\lambda_1+\lambda_2}$, hence eliminating the first two monomials (in the order used in \eqref{eq:examplefcoxno-2trick}) in $f_\cox(x^v)$. By looking at the degree with respect to the weight $(0,1,0,1)$, i.e., the combined degree in $x_2$ and $x_4$, we see that any contributing monomial must have degree at least $3\lambda_1+2\lambda_2$ with respect to this weight. In this way, we see that the last four monomials (again in the order used in \eqref{eq:examplefcoxno-2trick}) of $f_\cox(x^v)$ can not contribute. The remaining three monomials of $f_\cox(x^v)$ contribute to the desired monomial $x_1^{N+\lambda_1}x_2^{2N+\lambda_1+\lambda_2}x_3^{N+\lambda_1}x_4^{2\lambda_1+\lambda_2}$ in the product as follows: \begin{itemize}
    \item {$x_2^{\lambda_1+\lambda_2}x_3^{2\lambda_1}x_4^{2\lambda_1+\lambda_2}$:} This monomial can be complemented to the desired monomial  by choosing $N-\lambda_1$-times the monomial $x_2x_3$ and $N+\lambda_1$-times the monomial $x_1x_2$ when distributing the product $(x_1x_2+x_1x_4+x_2x_3+x_3x_4)^{2N}f_\cox(x^v)$. There are $\binom{2N}{N-\lambda_1}$ possible ways to choose these monomials from the $2N$ factors.
    \item {$x_1^{2\lambda_2}x_2^{\lambda_1+\lambda_2}x_4^{2\lambda_1+\lambda_2}$:} This monomial can be complemented to the desired monomial by choosing $N-\lambda_1$-times the monomial $x_1x_2$ and $N+\lambda_1$-times the monomial $x_2x_3$ when distributing the product $(x_1x_2+x_1x_4+x_2x_3+x_3x_4)^{2N}f_\cox(x^v)$. There are $\binom{2N}{N-\lambda_1}$ possible ways to choose these monomials from the $2N$ factors.
    \item {$- 2 x_1^{\lambda_1}x_2^{\lambda_1 + \lambda_2}x_3^{\lambda_1}x_4^{2\lambda_1 + \lambda_2}$:} This monomial can be complemented to the desired monomial by choosing $N$-times the monomial $x_1x_2$ and $N$-times the monomial $x_2x_3$ when distributing the product $(x_1x_2+x_1x_4+x_2x_3+x_3x_4)^{2N}f_\cox(x^v)$. There are $\binom{2N}{N}$ possible ways to choose these monomials from the $2N$ factors and the coefficient of the this monomial in $f_\cox(x^v)$ is $-2$.
\end{itemize}
Summing up all three cases we find that the coefficient of $x_1^{N+\lambda_1}x_2^{2N+\lambda_1+\lambda_2}x_3^{N+\lambda_1}x_4^{2\lambda_1+\lambda_2}$ in the product $(x_1x_2+x_1x_4+x_2x_3+x_3x_4)^{2N}f_\cox(x^v)$ is \begin{align*}
    \binom{2N}{N-\lambda_1}+\binom{2N}{N-\lambda_1}-2\binom{2N}{N}= 2\left(\binom{2N}{N-\lambda_1}-\binom{2N}{N}\right)
\end{align*}
as claimed. This is always negative since central binomial coefficients are maximal.
\end{example}

Next, we turn to Theorem~\ref{Thm:SparsePolya} and discuss possible generalizations. One might hope to allow different summands in the Minkowski sum $k\cdot A = A + \dots + A$ and still obtain a Pólya certificate, as was the case for the polynomials considered in \cite[Theorem 6.1]{SturmfelsTelek}. However, the claim does not hold in general.

\begin{example}
\label{Ex:MinkowskiSum}
     Consider the polynomial \begin{align*}
f=&2t_1^4t_2^2+t_1^3t_2^3+t_1^2t_2^4+2t_1^4t_2t_3-5t_1^3t_2^2t_3-2t_1t_2^4t_3+2t_1^4t_3^2+t_1^3t_2t_3^2+12t_1^2t_2^2t_3^2+t_1t_2^3t_3^2\\&+2t_2^4t_3^2-2t_1^3t_3^3-5t_1t_2^2t_3^3+2t_2^3t_3^3+t_1^2t_3^4+t_1t_2t_3^4+2t_2^2t_3^4
     \end{align*}
The Newton polytope of $f$ is the hexagon with vertices given by the columns of the matrix \begin{align*}
    \begin{pmatrix}
         4 & 2 & 4 & 0 & 2 & 0 \\ 2 & 4 & 0 & 4 & 0 & 2 \\ 0 & 0 & 2 & 2 & 4 & 4
    \end{pmatrix}.
\end{align*}
see Figure~\ref{FIG2}(c) for an illustration.
For $A = \supp(f)$, $f$ is strictly $A$-copositive by Theorem \ref{Thm:Main}, since the polynomial\begin{align*}
    \Big(\sum_{a\in A}t^a\Big)^3f
\end{align*}
has nonnegative coefficients.
The polynomial $f$ factors into two irreducible polynomials\begin{align*}
    f_1&=2t_1^2+t_1t_2+t_2^2-2t_1t_3+t_2t_3+t_3^2, \\
    f_2&=t_1^2t_2^2+t_1^2t_2t_3-2t_1t_2^2t_3+t_1^2t_3^2+t_1t_2t_3^2+2t_2^2t_3^2.
\end{align*}
Hence the Newton polytope of $f$ is the Minkowski sum of two triangles, namely the two Newton polytopes of $f_1$ and $f_2$. We get two possible candidates for Pólya multipliers of $f$ by summing the monomials in $f_1$ and $f_2$ respectively \begin{align*}
    g_1&=t_1^2+t_1t_2+t_2^2+t_1t_3+t_2t_3+t_3^2, \\
    g_2&=t_1^2t_2^2+t_1^2t_2t_3+t_1t_2^2t_3+t_1^2t_3^2+t_1t_2t_3^2+t_2^2t_3^2.
\end{align*}
However, by looking at the highest degree in the variable $t_2$, we easily see that the polynomial $g_1^Nf$ has negative coefficient $-2$ in front of the monomial $t_1t_2^{2N+4}t_3$ for all $N\in \NN$. Similarly by looking at the degree zero part in the variable $y$, we find that in $g_2^Nf$, the monomial $t_1^{2N+3}t_3^{2N+3}$ always has negative coefficient $-2$. We conclude that in general Minkowski summands do not suffice as Pólya multipliers.
\end{example}

\section{Symanzik polynomials}
\label{sec:Applications}
\label{sec:Feynman}

The convergence of Feynman integrals is closely linked to copositivity of sparse polynomials.
In this section, we briefly review this relation and showcase how our sparse version of Pólya's method (Theorem \ref{Thm:SparsePolya}) can be applied to the study of convergence of Feynman integrals, extending beyond the cases considered in \cite{SturmfelsTelek}.

Feynman integrals are fundamental objects in particle physics, serving as building blocks of scattering amplitudes, which describe the probabilities of specific outcomes in scattering experiments.
For a more detailed overview of their role in particle physics, we refer the reader to \cite{Weinzierl}.
Here, we recall only the absolutely necessary mathematical background required to understand our results.

The possible interactions between particles are represented by \emph{Feynman diagrams}, which are connected graphs $G$ equipped with some kinematic data. A \emph{Feynman integral} (in the Feynman parameter representation) is defined as
\begin{align}
\label{Eq:FeynmParamInt}
    I_G(z) \,\,= \int_{\mathbb{P}_{>0}^{n-1}} \frac{\left(\,\prod_{i=1}^n x_i^{\nu_i} \right)\mathcal{U}(x)^{|\nu|-(\ell+1) D /2}}{\mathcal{F}(x)^{|\nu| - \ell D / 2}} \; \text{\footnotesize$ \left(\sum_{i=1}^n(-1)^{n-i} \frac{\mathrm{~d} x_1}{x_1} \wedge \cdots \wedge \frac{\widehat{\mathrm{~d} x_i}}{x_i} \wedge \cdots \wedge \frac{\mathrm{~d} x_n}{x_n}\right)$},
\end{align}
where $D,\ell,n \in \mathbb{N}$, $\nu_1, \dots, \nu_n \in \mathbb{R}$, $|\nu| = \nu_1 + \dots + \nu_n$ and $\mathcal{U}(x), \mathcal{F}(x)$ are polynomials associated to the Feynman diagram.
We will elaborate on the construction of these polynomials and the kinematic parameters in more detail below. First, however, we recall the following theorem, which relates copositivity to the convergence of Feynman integrals.

\begin{theorem}
\label{Prop:ConvergInER}
\cite[Theorem 3]{Borinsky} 
The Feynman integral in \eqref{Eq:FeynmParamInt} converges if
\begin{itemize}
\item[(i)] $\mathcal{F}_z$ is strictly $A$-copositive with respect to $A = \supp(\mathcal{F}_z)$, 
    \item[(ii)] $\Newt\, \! \left(  \left(\,\prod_{i=1}^n x_i^{\nu_i} \right) \mathcal{U}(x)^{|\nu|-(\ell+1) D /2} \right)
\, \subseteq \,\relint  \Newt \!  \left(\mathcal{F}(x)^{|\nu| - \ell D / 2} \right)$.
\end{itemize}
\end{theorem}

Condition (ii) in Theorem~\ref{Prop:ConvergInER} can be achieved by choosing $D,\nu_1,\dots,\nu_n$ appropriately. In this section, we focus on certifying condition (i).

We recall the construction of the polynomials $\mathcal{U}(x)$ and $\mathcal{F}(x)$.
Let $G$ be a connected graph. The letter $\ell$ in \eqref{Eq:FeynmParamInt} denotes the number of independent cycles in $G$, that is,
\[ \ell = \# \text{ edges of $G$} \; - \; \# \text{ vertices of $G$} \; + \; 1. \]
An edge of $G$ is called an \emph{external edge} if one of its vertices is attached to no other edge of $G$. Otherwise, we call an edge an \emph{internal edge}.
For example, the \emph{banana diagram}
\begin{equation}
\begin{aligned}
\label{Eq:Banana}
\begin{tikzpicture}[scale=0.3]
    \node[] (P4) at (4.6, 2) {\footnotesize $p_4$};
     \node[] (P3) at (4.6, -2) {\footnotesize $p_3$};
    \node[] (P1) at (-4.6, 2) {\footnotesize $p_1$};
     \node[] (P2) at (-4.6, -2) {\footnotesize $p_2$};
    \node[fill={black}, circle, inner sep=1.5pt]
    (A) at (2.5, 0) {};
     \node[fill={black}, circle, inner sep=1.5pt] (B) at (-2.5, 0) {};
     \node[fill=lightgray, circle, inner sep=1.5pt] (A1) at (5.5, 1.8) {};
      \node[fill=lightgray, circle, inner sep=1.5pt]  (A2) at (5.5, -1.8) {};
    \node[fill=lightgray, circle, inner sep=1.5pt] (B1) at (-5.5, 1.8) {};
      \node[fill=lightgray, circle, inner sep=1.5pt] (B2) at (-5.5, -1.8) {};
     \draw[bend left=70] (A) to (B) node[midway, above=-45 pt] {\Huge$m_3$};
      \draw[] (A) to (B) node[midway, above=2 pt]  {\Huge$m_2$};
     \draw[bend right=70] (A) to (B) node[midway, above=45 pt]  {\Huge$m_1$};
    \draw (A) -- (A1);
    \draw (A) -- (A2);
       \draw (B) -- (B1);
    \draw (B) -- (B2);
\end{tikzpicture}
\end{aligned}
\end{equation}

has $4$ external and $3$ internal edges and $\ell = 2$. In general, we denote by $N$ the number of external edges, and by $n$ the number of internal edges of a graph $G$. Each external edge is assigned a \emph{momentum vector} $p_i \in \mathbb{R}^D$, while the internal edges carry kinematic parameters $m_1,\dots,m_n$, called \emph{internal masses}, together with variables $x_1,\dots,x_n$. The \emph{first Symanzik polynomial} is defined as
\begin{align*}
    \mathcal{U}(x) \,\,:=\,\, \sum_{T \in \mathcal{T}} \prod_{e \notin T} x_e,
\end{align*}
where the sum is over the set of all spanning trees $\mathcal{T}$ of the graph $G$. These are connected subgraphs without
cycles that contain all vertices of $G$. The \emph{second Symanzik polynomial} equals
\begin{align}
\label{Eq:F}
    \mathcal{F}(x)\,\, := \sum_{\{T_1,T_2\} \in \mathcal{W}}\!\biggl( \,\sum_{i \in I_{T_1}} \sum_{j \in I_{T_2}} k_{ij} 
    \biggr) \!\! \prod_{e \notin T_1 \sqcup T_2} \!\! \!x_e \,
    \,+ \,\,\biggl(\,\sum_{e=1}^n m_e x_e \biggr)\, \mathcal{U}(x).
\end{align}
The first sum is over the set of all \emph{spanning 2-forests}, that is, each $\{T_1,T_2\} \in \mathcal{W}$ consists of two disjoint connected subgraphs of $G$ that contain no cycles, and together $T_1 \cup T_2$ covers all vertices of $G$. The sets $I_{T_1}, I_{T_2} \subseteq [N]$ record the indices of the external edges belonging to $T_1$ and $T_2$, respectively.
Moreover, $k_{ij}$ denotes the Minkowski scalar product of the momentum vectors $p_i, \, p_j$, that is, $k_{ij} = p_{i1}p_{j1} -  p_{i2}p_{j2}-\dots -  p_{iD}p_{jD}$. The second Symanzik polynomial is homogeneous of degree $\ell+1$, but each variable appears with exponent at most $2$. Thus, these polynomials are sparse, as many monomials of degree $\ell+1$ are absent.
\begin{example}
The banana diagram \eqref{Eq:Banana} has three spanning trees and one spanning $2$-forest
\begin{center}
\begin{minipage}{0.22\textwidth}
\centering
\begin{tikzpicture}[scale=0.3]
    \node[] (P4) at (4.6, 2) {\footnotesize $p_4$};
     \node[] (P3) at (4.6, -2) {\footnotesize $p_3$};
    \node[] (P1) at (-4.6, 2) {\footnotesize $p_1$};
     \node[] (P2) at (-4.6, -2) {\footnotesize $p_2$};
    \node[fill={black}, circle, inner sep=1.5pt]
    (A) at (2.5, 0) {};
     \node[fill={black}, circle, inner sep=1.5pt] (B) at (-2.5, 0) {};
     \node[fill=lightgray, circle, inner sep=1.5pt] (A1) at (5.5, 1.8) {};
      \node[fill=lightgray, circle, inner sep=1.5pt]  (A2) at (5.5, -1.8) {};
    \node[fill=lightgray, circle, inner sep=1.5pt] (B1) at (-5.5, 1.8) {};
      \node[fill=lightgray, circle, inner sep=1.5pt] (B2) at (-5.5, -1.8) {};
     \draw[bend right=70] (A) to (B) node[midway, above=45 pt]  {\Huge$m_1$};
    \draw (A) -- (A1);
    \draw (A) -- (A2);
       \draw (B) -- (B1);
    \draw (B) -- (B2);
\end{tikzpicture}
\end{minipage}
\hspace{8pt}
\begin{minipage}{0.22\textwidth}
\centering
\begin{tikzpicture}[scale=0.3]
    \node[] (P4) at (4.6, 2) {\footnotesize $p_4$};
     \node[] (P3) at (4.6, -2) {\footnotesize $p_3$};
    \node[] (P1) at (-4.6, 2) {\footnotesize $p_1$};
     \node[] (P2) at (-4.6, -2) {\footnotesize $p_2$};
    \node[fill={black}, circle, inner sep=1.5pt]
    (A) at (2.5, 0) {};
     \node[fill={black}, circle, inner sep=1.5pt] (B) at (-2.5, 0) {};
     \node[fill=lightgray, circle, inner sep=1.5pt] (A1) at (5.5, 1.8) {};
      \node[fill=lightgray, circle, inner sep=1.5pt]  (A2) at (5.5, -1.8) {};
    \node[fill=lightgray, circle, inner sep=1.5pt] (B1) at (-5.5, 1.8) {};
      \node[fill=lightgray, circle, inner sep=1.5pt] (B2) at (-5.5, -1.8) {};
      \draw[] (A) to (B) node[midway, above=2 pt]  {\Huge$m_2$};
    \draw (A) -- (A1);
    \draw (A) -- (A2);
       \draw (B) -- (B1);
    \draw (B) -- (B2);
\end{tikzpicture}
\end{minipage}
\hspace{8pt}
\begin{minipage}{0.22\textwidth}
\centering
\begin{tikzpicture}[scale=0.3]
    \node[] (P4) at (4.6, 2) {\footnotesize $p_4$};
     \node[] (P3) at (4.6, -2) {\footnotesize $p_3$};
    \node[] (P1) at (-4.6, 2) {\footnotesize $p_1$};
     \node[] (P2) at (-4.6, -2) {\footnotesize $p_2$};
    \node[fill={black}, circle, inner sep=1.5pt]
    (A) at (2.5, 0) {};
     \node[fill={black}, circle, inner sep=1.5pt] (B) at (-2.5, 0) {};
     \node[fill=lightgray, circle, inner sep=1.5pt] (A1) at (5.5, 1.8) {};
      \node[fill=lightgray, circle, inner sep=1.5pt]  (A2) at (5.5, -1.8) {};
    \node[fill=lightgray, circle, inner sep=1.5pt] (B1) at (-5.5, 1.8) {};
      \node[fill=lightgray, circle, inner sep=1.5pt] (B2) at (-5.5, -1.8) {};
     \draw[bend left=70] (A) to (B) node[midway, above=-45 pt] {\Huge$m_3$};
    \draw (A) -- (A1);
    \draw (A) -- (A2);
       \draw (B) -- (B1);
    \draw (B) -- (B2);
\end{tikzpicture}
\end{minipage}
\hspace{8pt}
\begin{minipage}{0.22\textwidth}
\centering
\begin{tikzpicture}[scale=0.3]
    \node[] (P4) at (4.6, 2) {\footnotesize $p_4$};
     \node[] (P3) at (4.6, -2) {\footnotesize $p_3$};
    \node[] (P1) at (-4.6, 2) {\footnotesize $p_1$};
     \node[] (P2) at (-4.6, -2) {\footnotesize $p_2$};
    \node[fill={black}, circle, inner sep=1.5pt]
    (A) at (2.5, 0) {};
     \node[fill={black}, circle, inner sep=1.5pt] (B) at (-2.5, 0) {};
     \node[fill=lightgray, circle, inner sep=1.5pt] (A1) at (5.5, 1.8) {};
      \node[fill=lightgray, circle, inner sep=1.5pt]  (A2) at (5.5, -1.8) {};
    \node[fill=lightgray, circle, inner sep=1.5pt] (B1) at (-5.5, 1.8) {};
      \node[fill=lightgray, circle, inner sep=1.5pt] (B2) at (-5.5, -1.8) {};
    \draw (A) -- (A1);
    \draw (A) -- (A2);
       \draw (B) -- (B1);
    \draw (B) -- (B2);
\end{tikzpicture}
\end{minipage}
     \end{center}
     The first and second Symanzik polynomials for the banana diagram are
     \begin{equation}
     \label{Eq:BananaF}
     \begin{aligned}
        \mathcal{U}(x) &= x_2x_3 + x_1x_3 + x_1x_2 \\
        \mathcal{F}(x) &=   (k_{13} + k_{14} +k_{23} + k_{24}) x_1x_2x_3 + (m_1x_1 + m_2x_2 + m_3x_3)(x_2x_3 + x_1x_3 + x_1x_2 ) \\
         &=(k_{13} + k_{14} +k_{23} + k_{24} + m_1 + m_2 + m_3) x_1x_2x_3 \\
         &+ m_2 x_2^2 x_3  + m_3 x_2 x_3^2 + m_1 x_1^2x_2 + m_1 x_1^2x_3  + m_2 x_1x_2^2 + m_3 x_1x_3^2.
     \end{aligned}
     \end{equation}
   The Newton polytope of $\mathcal{F}$ lives in the ambient space $\mathbb{R}^3$. However, as $\mathcal{F}$ is homogeneous, its Newton polytope lives in the affine hyperplane $x_1 + x_2 + x_3 = 3$, and it has dimension two. For an illustration, we refer to Figure~\ref{FIG1}(a).  
\end{example}

\begin{figure}[t]
\centering
\begin{minipage}[h]{0.34\textwidth}
\centering
\includegraphics[scale=0.4]{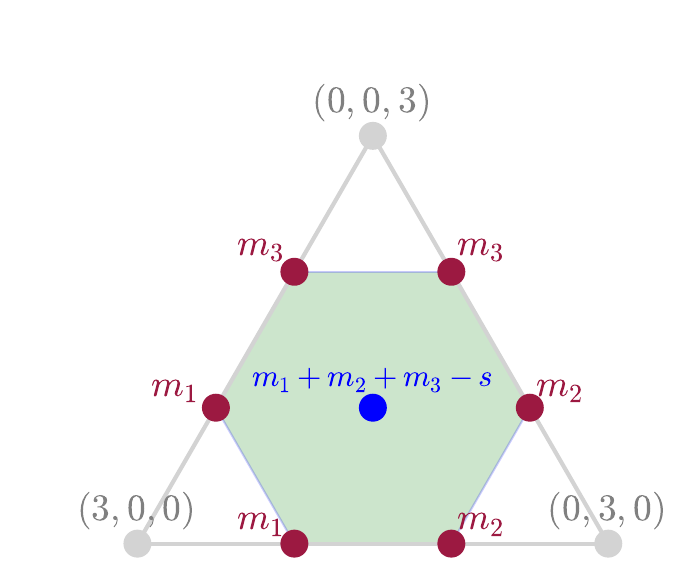}

{\small (a)}
\end{minipage}
\hspace{10 pt}
\begin{minipage}[h]{0.4\textwidth}
\centering
\includegraphics[scale=0.4]{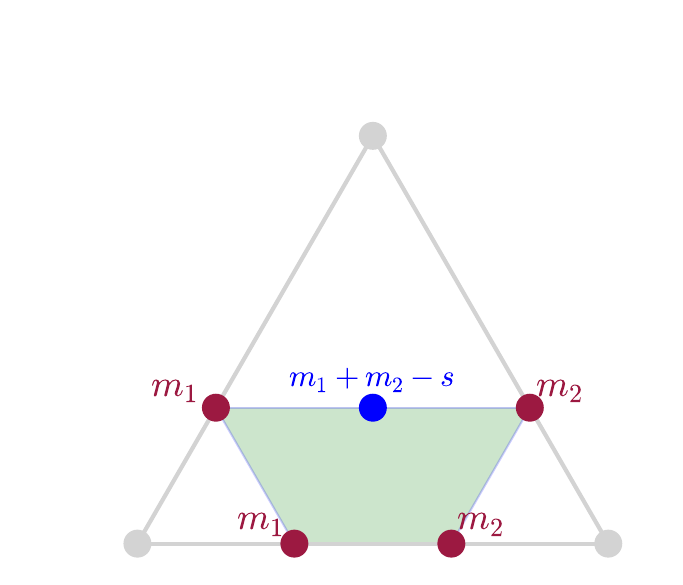}

{\small (b)}
\end{minipage}
\caption{{\small  (a) Newton polytope of the second Symanzik polynomial $\mathcal{F}$ of the banana diagram \eqref{Eq:BananaF}, $s = -(k_{13} + k_{14} +k_{23} + k_{24})$.  (b) depicts the Newton polytope of $\mathcal{F}$ after setting $m_3 = 0$.}}\label{FIG1}
\end{figure}

In the physics context, the coefficients of $\mathcal{F}$ may satisfy certain linear relations. For example, momentum conservation implies that $k_{i1} + \dots + k_{iN} = 0$ for each $i \in [N]$. Another common assumption is that some of the internal masses are zero. To accommodate these scenarios, we consider a family of polynomials with the following parametrization:
\begin{align}
\label{Eq:Fzm}
    \mathcal{F}_{z,m}(x)\,\, := \sum_{\{T_1,T_2\} \in \mathcal{W}} L_{T_1,T_2}(z) \!\! \prod_{e \notin T_1 \sqcup T_2} \!\! \!x_e \,
    \,+ \,\,\biggl(\,\sum_{e\in\mathcal{E}}m_e x_e \biggr)\, \mathcal{U}(x),
\end{align}
where $L\colon \mathbb{R}^K \to \mathbb{R}^{\mathcal{W}}$, and $\mathcal{E} \subseteq [n]$.
Let $A_G
\subseteq \mathbb{Z}^n$ be the support of $ \mathcal{F}_{z,m}(x)$ for generic values of $(z,m) \in \mathbb{R}^K \times \mathbb{R}^{\mathcal{E}}$.
Because of the linear relations among the coefficients, the set of polynomials 
$\{\, \mathcal{F}_{z,m}(x) \, | \, (z,m) \in \mathbb{R}^K \times \mathbb{R}^{\mathcal{E}} \, \}$ 
forms a linear subspace of $\mathbb{R}[x_1^\pm,\dots,x_n^\pm]_{A_G}$.
We call this subspace the \emph{kinematic space} 
$\mathcal{K}_G$.

In \cite[Section 3.1]{HennRaman}, the intersection $\mathcal{K}_G \cap \interior(\mathcal{C}_{A_G})$ is referred to as the \emph{Euclidean region} of the Feynman diagram $G$. Its relevance in physics comes from the fact that if the second Symanzik polynomial lies in  $\mathcal{K}_G \cap \interior(\mathcal{C}_{A_G})$, then the convergence of the Feynman integral \eqref{Eq:FeynmParamInt} can be ensured using Theorem~\ref{Prop:ConvergInER}. In the remainder of this section, we provide an easily checkable necessary and sufficient condition for the Euclidean region to be non-empty (Proposition~\ref{Prop:IntersectCoposCone}).

To that end, note that support set $A_G$ is contained in the union of the following two sets
\begin{align}
\label{Eq:supportFz}
 A^1_G = \Big\{ \,e_p \,+\!\! \sum_{j\in [n] \backslash T } \!e_j \,\mid \,T \in \mathcal{T} ,\,p \in T \Big\},
\quad 
    A^2_G\, =\, \Big\{ 2e_q \,+\!\!\! \sum_{j\in [n]\backslash (T \cup \{q\})}\!\!\!
    \!\! e_j \,\mid\, T \in \mathcal{T} ,\; q \not\in T \Big\}.
 \end{align}
The next technical lemma will be used in the proof of Proposition~\ref{Prop:IntersectCoposCone}.
\begin{lemma}
\label{Lemma:NotAVertex}
Let $A_G \subseteq A_G^1 \cup A_G^2$ be the support of $\mathcal{F}_{z,m}$ in \eqref{Eq:F} for generic $(z,m)$.
If for $a \in A_G \cap A_G^1$ the coefficient of $x^a$ does not involve any $m_e$'s, then $a$ is a vertex of $\conv(A_G)$.
\end{lemma}

\begin{proof}
Since $a \in A^1_G$, there exists a spanning tree $T$ of $G$ and $p \in T$ such that  $a = e_p +  \sum_{j\in [n] \backslash T } \!e_j$. 
We prove the statement by contraposition.
Assume that $a$ is not a vertex of $\conv(A_G)$. Then there exists $b_1,\dots,b_k \in A_G \setminus \{a\}$ and $\lambda_1,\dots,\lambda_k > 0$ such that $a = \sum_{i=1}^k\lambda_i b_i$ and $\sum_{i=1}^k\lambda_i =1$. 

For a vector $v \in \mathbb{R}^n$, we write $\supp(v) = \{ i \in [n] \mid v_i \neq 0 \}$. Since the entries of $b_1, \dots, b_k$ are nonnegative, we have $\supp(b_i) \subseteq \supp(a) = \{p\}\cup ([n] \setminus T) $ for each $i \in [k]$. 
Since $b_1 \neq a$ and the sum of the coordinates of $b_1$ is equal to $\ell +1 =\# (\{p\} \cup [n]\setminus T)$, there exists $e \in \{p\}\cup ([n] \setminus T)$ such that $e \notin \supp(b_1)$, and $b_1 \in A_{G}^2 \cap A_G$.
This implies that there exists a spanning tree $T_1$ such that $\supp(b_1) = [n] \setminus T_1$.
Thus, we have $e \in T_1$, $\supp(a) = \{e \} \cup ([n] \setminus T_1)$ and $m_e \neq 0$.
Using \eqref{Eq:F}, we conclude that the coefficient of $x^a$ involves $m_{e}$, which concludes the proof.
\end{proof}

\begin{proposition}
\label{Prop:IntersectCoposCone}
Let $G$ be a Feynman diagram and let $A_G$ be the support of $\mathcal{F}_{z,m}$ in \eqref{Eq:F} for generic $(z,m) \in \RR^K \times \RR^{\mathcal{E}}$.
The Euclidean region $\mathcal{K}_G \cap \interior \mathcal{C}_{A_G}$ is nonempty if and only if $\mathcal{K}_G \cap \mathbb{R}_{> 0}[x_1^\pm, \dots,x_n^\pm]_{A_G} \neq  \emptyset$.
\end{proposition}

\begin{proof}
The only if part is obvious, since $\mathbb{R}_{>0}[x_1^\pm, \dots,x_n^\pm]_{A_G} \subseteq \interior(\mathcal{C}_{A_G})$.
To show the other implication, let $\mathcal{F}_{z,m} \in \mathcal{K}_G \cap \interior \mathcal{C}_{A_G}$.
If $a \in A_G$ is a vertex of $\conv(A_G)$, then the coefficient of the monomial $x^a$ must be positive as $\mathcal{F}_{z,m} \in \interior (\mathcal{C}_{A_G})$.
If $a \in A_G$ is not a vertex, then by Lemma~\ref{Lemma:NotAVertex} the coefficient of $x^a$ involves at least one $m_e$. By choosing $m_e$ sufficiently large, the coefficient of $x^a$ becomes positive, which completes the proof.
\end{proof}

We conclude this section with two examples. We begin by showing that the Euclidean region might be empty, which is known to happen for non-planar diagrams~\cite{HennSmirnovSmirnov,HennRaman}.
\begin{example}
    Consider the non-planar double box diagram
   \begin{equation*}
    \begin{aligned}
        \begin{tikzpicture}[scale=0.6]
    \node[fill=black, circle, inner sep=1.5pt] (C) at (0, 1) {};
    \node[fill=black, circle, inner sep=1.5pt] (D) at (0, -1) {};
        \node[fill=black, circle, inner sep=1.5pt] (E) at (2, 1) {};
    \node[fill=black, circle, inner sep=1.5pt] (F) at (2, -1) {};
        \node[fill=black, circle, inner sep=1.5pt] (A) at (-2, 1) {};
    \node[fill=black, circle, inner sep=1.5pt] (B) at (-2, -1) {};
    \node[fill=gray, circle, inner sep=1.5pt] (U) at (3, 1.7) {};
    \node[fill=gray, circle, inner sep=1.5pt] (V) at (-3, 1.7) {};
    \node[fill=gray, circle, inner sep=1.5pt] (W) at (3, -1.7) {};
    \node[fill=gray, circle, inner sep=1.5pt] (T) at (-3, -1.7) {};

        \node[] (P) at (1, 0) {};
    \node[] (Q) at (1, 0) {};
  
    \draw (A) -- (B);
    \draw (A) -- (C);
    \draw (B) -- (D);
    \draw (C) -- (F);
    \draw (P) -- (D);
    \draw (E) -- (Q);
    \draw (C) -- (E);
    \draw (D) -- (F);
    \draw (E) -- (U);
    \draw (A) -- (V);
    \draw (F) -- (W);
    \draw (T) -- (B);

    \node[] (m4) at (0.2, 0.25) {\footnotesize $m_4$};
    \node[] (m5) at (1.85, 0.25) {\footnotesize $m_7$};
    \node[] (m1) at (-2.4, 0) {\footnotesize $m_2$};
    \node[] (m_2) at (-1, -1.25) {\footnotesize $m_1$};
    \node[] (m3) at (-1, 1.25) {\footnotesize $m_3$};
    \node[] (m6) at (1, -1.25) {\footnotesize $m_6$};
    \node[] (m7) at (1, 1.25) {\footnotesize $m_5$};
    \node[] (P1) at (2.6, 1.8) {\footnotesize $p_3$};
    \node[] (P1) at (-2.6, 1.8) {\footnotesize $p_2$};
    \node[] (P1) at (2.6, -1.8) {\footnotesize $p_4$};
    \node[] (P1) at (-2.6, -1.8) {\footnotesize $p_1$};
\end{tikzpicture}
    \end{aligned}
    \end{equation*}
    with all external masses set to zero, i.e., $k_{11} = k_{22} = k_{33} = k_{44} = 0$ (cf. \eqref{Eq:F}). In addition, we set the internal masses $m_3, m_4, m_5, m_6, m_7$  to zero as well. Under the assumption of momentum conservation, the second Symanzik polynomial is given by
    \begin{align*}
        &\mathcal{F}_{z,m}(x) = m_1x_1^2x_6 + m_1x_1^2x_7 + m_1x_4x_1^2 + m_1x_5x_1^2 + m_2x_2^2x_6 + m_2x_2^2x_7 + m_2x_4x_2^2 \\ 
        &+ m_2x_5x_2^2 + m_1x_1x_6x_7 + (m_1 + m_2)x_2x_1x_6 + (m_1 + m_2)x_2x_1x_7     + (m_1 - s)x_3x_1x_6 + 
          (m_1-s)x_3x_1x_7    \\
        &+ m_1x_4x_1x_7  + (m_1+m_2)x_4x_2x_1 + (m_1 -s)x_4x_3x_1 + (m_1-s)x_4x_5x_1 + m_1x_5x_1x_6 + (m_1 + m_2) x_5x_2x_1 \\
        &+ (m_1 - s)x_5x_3x_1     
        + m_2x_2x_3x_6 
        + m_2x_2x_3x_7 
        + m_2x_2x_6x_7 
        + m_2x_4x_2x_3 
        + (m_2 - t)x_4x_2x_7 \\
        &+ m_2x_4x_5x_2 
      + m_2x_5x_2x_3
      +( m_2 + s + t)x_5x_2x_6  - sx_3x_6x_7 
    \end{align*}
One easily checks that for $s < 0$ and $m_1, m_2 > 0$ sufficiently large $\mathcal{F}_{z,m}(x)$  has only positive coefficients. Thus, the kinematic space $\mathcal{K}_G$ intersects $\interior(\mathcal{C}_{A_G})$ and one can use Theorem~\ref{Thm:SparsePolya} to make containment in the Euclidean region $\mathcal{K}_G \cap \interior(\mathcal{C}_{A_G})$ manifest.

If we additionally set $m_2 = 0$, the situation changes. The monomials $x_2x_5x_6, \, x_3x_6x_7, \, x_2x_4x_7$ have coefficients $s+t,-s$ and $-t$ respectively. In this case, $\mathcal{F}_{z,m}(x)$ has both positive and negative coefficients for all values of $m_1,s,t$. 
Using Proposition~\ref{Prop:IntersectCoposCone}, we conclude that the corresponding Euclidean region is empty.
\end{example}

Both Theorem~\ref{Thm:SparsePolya} and \cite[Theorem 6.1]{SturmfelsTelek} provide a way to make containment in the Euclidean region manifest.
In the next example, we compare these two approaches.

\begin{example}
    The second Symanzik polynomial of the banana diagram \eqref{Eq:Banana} is given by
    \begin{align*}
         \mathcal{F}_{z,m}(x) =(m_1 + m_2 + m_3 - s) x_1x_2x_3 + m_2 x_2^2 x_3  + m_3 x_2 x_3^2 + m_1 x_1^2x_2 + m_1 x_1^2x_3  + m_2 x_1x_2^2 + m_3 x_1x_3^2.
    \end{align*}
   Here, we set  $k_{13} + k_{14} +k_{23} + k_{24} = -s$, cf. \eqref{Eq:BananaF}. If all internal masses are non-zero, both Theorem~\ref{Thm:SparsePolya} and \cite[Theorem 6.1]{SturmfelsTelek} can be used to check containment in the Euclidean region. In this case, one verifies whether
   \begin{align}
   \label{Eq:SparseMethod}
       &(x_1x_2x_3 + x_2^2 x_3  + x_2 x_3^2 + x_1^2x_2 + x_1^2x_3  + x_1x_2^2 + x_1x_3^2)^N \mathcal{F}_{z,m}(x) \quad \text{or} \\
       \label{Eq:ZerosMethod}
        &(x_1 + x_2 + x_3)^N \mathcal{F}_{z,m}(x)
   \end{align}
has only positive coefficients for some $N \in \mathbb{N}$.
For example, when $m_1 = m_2 = m_3 = 1$ and $s= 8.97$ the smallest $N$ for which \eqref{Eq:SparseMethod} has only positive coefficients is $N=232$, while for \eqref{Eq:ZerosMethod} it is $N=596$. Although the exponent is much larger in the second case, the time required to compute the product is approximately the same, and the resulting polynomials have roughly the same number of monomials: $2 732 802$ for \eqref{Eq:SparseMethod} and $3 056 236$ for \eqref{Eq:ZerosMethod}.

We modify the example and assume that $m_3 = 0$. Now, the second Symanzik polynomial is 
    \begin{align*}
         \mathcal{F}_{z,m}(x) =(m_1 + m_2 - s) x_1x_2x_3 + m_2 x_2^2 x_3   + m_1 x_1^2x_2 + m_1 x_1^2x_3  + m_2 x_1x_2^2.
    \end{align*}
    Since the Newton polytope of $ \mathcal{F}_{z,m}(x)$ has changed (see Figure~\ref{FIG1}(b)), \cite[Theorem 6.1]{SturmfelsTelek} does not apply in this case. However, we can still use Theorem~\ref{Thm:SparsePolya} and check whether
     \begin{align*}
       ( x_1x_2x_3 + x_2^2 x_3   +  x_1^2x_2 + x_1^2x_3  + x_1x_2^2)^N  \mathcal{F}_{z,m}(x) 
    \end{align*}
    has positive coefficients for some $N \in \mathbb{N}$.
\end{example}

\printbibliography

@article{CastlePowersReznick,
title = {Pólya’s Theorem with zeros},
journal = {J. Symb. Comput.},
volume = {46},
number = {9},
pages = {1039-1048},
year = {2011},
author = {Castle, M. and Powers, V. and Reznick, B.}
}

@book{MaclaganSturmfels,
  title={Introduction to Tropical Geometry},
  author={Maclagan, D. and Sturmfels, B.},
  isbn={9780821851982},
  lccn={2014036141},
  series={Graduate Studies in Mathematics},
  year={2015},
  publisher={American Mathematical Society}
}

@article{Hilbert,
title = {Mathematical Problems},
journal = {Bull. Amer. Math. Soc.},
volume = {8},
pages = {437–479},
year = {1902},
author = {Hilbert, D.}
}

@article{Artin,
title = {Über die {Z}erlegung definiter {F}unktionen in {Q}uadrate},
journal = {Abh. Math. Semin. Univ. Hambg.},
volume = {5},
pages = {100–115},
year = {1927},
author = {Artin, E.}
}

@article{Motzkin,
      title={Copositive quadratic forms}, 
      author={Motzkin, T.~S.},
      year={1952},
      journal={National Bureau of Standards Report 1818},
      pages = {11-12}
}

@article{Polya,
 author = {P{\'o}lya, G.},
 title = {{\"U}ber positive {Darstellung} von {Polynomen}},
 year = {1928},
 journal = {Vierteljahrsschrift {Z{\"u}rich}},
 volume = {73},
 pages = {141-145}
}

@article{PowersReznick,
title = {A new bound for {P}ólya's Theorem with applications to polynomials positive on polyhedra},
journal = {J. Pure Appl. Algebra},
volume = {164},
number = {1},
pages = {221-229},
year = {2001},
author = {Powers, V. and Reznick, B.}
}

@article{BomzeDur,
title = {On Copositive Programming and Standard Quadratic Optimization Problems},
journal = {J. Glob. Optim.},
volume = {18},
pages = {301–320},
year = {2000},
author = {Bomze, I.~M. and Dür, M. and de~Klerk, E. and Roos, C. and Quist, A.~J. and Terlaky, T. },
}

@InProceedings{Dur,
author={D{\"u}r, M.},
editor={Diehl, M. and Glineur, F. and Jarlebring, E. and Michiels, W.},
title={Copositive Programming -- a Survey},
booktitle={Recent Advances in Optimization and its Applications in Engineering},
year={2010},
pages={3--20}
}

@article{BodirskyKummerThom,
author = {Bodirsky, M. and Kummer, M. and Thom, A.},
year = {2024},
title = {Spectrahedral shadows and completely positive maps on real closed fields},
journal = {J. Eur. Math. Soc.},
author = {Bodirsky,  Manuel and Kummer,  Mario and Thom,  Andreas},
note = {published online first}
}

@article{deLoeraSantos,
title = {An effective version of {P}ólya's theorem on positive definite forms},
journal = {J. Pure Appl. Algebra.},
volume = {108},
number = {3},
pages = {231-240},
year = {1996},
issn = {0022-4049},
author = {{de Loera}, J.~A. and Santos, F.},
}

@article{MokTo,
author = {Mok, H.-N. and To, W.-K.},
year = {2008},
volume = {24},
number = {4},
pages = {524-544},
title = {Effective {P}ólya semi-positivity for non-negative polynomials on the simplex},
journal = {J. Complex.},
}

@book{Parrilo,
publisher = {Ph.D. thesis, California Institute of Technology, Pasadena,CA},
author = {Parrilo, P.~A.},
title = {Structured Semideﬁnite Programs and Semi-algebraic Geometry Methods in Robustness and Optimization},
year = {2000},
comment = {Available online at: https://www.mit.edu/~parrilo/pubs/files/thesis.pdf}
}

@book{Marshall,
AUTHOR = {Marshall, Murray},
     TITLE = {Positive polynomials and sums of squares},
    SERIES = {Mathematical Surveys and Monographs},
    VOLUME = {146},
 PUBLISHER = {American Mathematical Society, Providence, RI},
      YEAR = {2008},
}

@book{Vargas,
publisher = {Ph.D. thesis,  Tilburg University},
author = {Vargas, L.~F.},
title = {Sum-of-Squares Representations for Copositive Matrices and Independent Sets in Graphs},
year = {2023},
comment = {Available online at: https://pure.uvt.nl/ws/portalfiles/portal/81931069/Proefschrift_Luis_Vargas_digitaal.pdf}
}

@article{DresslerMurray,
  author    = {Dressler, M. and Murray, R.},
  title     = {Algebraic Perspectives on Signomial Optimization},
  journal   = {SIAM J. Appl. Algebra Geom.},
  volume    = {6},
  number    = {4},
  pages     = {650--684},
  year      = {2022},
  doi       = {10.1137/21M1462568},
  url       = {https://epubs.siam.org/doi/10.1137/21M1462568}
}

@article{Delzell,
title = {Impossibility of extending {P}ólya’s theorem to “forms” with arbitrary real exponents},
journal = {J. Pure Appl. Algebra.},
volume = {212},
number = {12},
pages = {2612-2622},
year = {2008},
issn = {0022-4049},
doi = {https://doi.org/10.1016/j.jpaa.2008.04.006},
url = {https://www.sciencedirect.com/science/article/pii/S002240490800090X},
author = {Delzell, C.~N.},
}

@article{ChandrasekaranShah,
author = {Chandrasekaran, V. and Shah, P.},
title = {Relative Entropy Relaxations for Signomial Optimization},
journal = {SIAM J. Optim.},
volume = {26},
number = {2},
pages = {1147-1173},
year = {2016},
doi = {10.1137/140988978},
}

@inproceedings{PowersReznick2006,
  author    = {Powers, V. and Reznick, B.},
  title     = {A Quantitative {P}ólya’s Theorem with Corner Zeros},
  booktitle = {Proceedings of the 2006 International Symposium on Symbolic and Algebraic Computation (ISSAC)},
  editor    = {Dumas, J.-G.},
  year      = {2006},
  pages     = {285--290},
  publisher = {ACM Press},
}

@misc{WangJainiYuPoupart,
      title={A {P}ositivstellensatz for Conditional {SAGE} Signomials}, 
      author={Allen Houze Wang and Priyank Jaini and Yaoliang Yu and Pascal Poupart},
      year={2020},
      eprint={2003.03731},
      archivePrefix={arXiv},
}

@misc{telen2022introductiontoricgeometry,
      title={Introduction to Toric Geometry}, 
      author={Simon Telen},
      year={2022},
      eprint={2203.01690},
      archivePrefix={arXiv},
}

@book{CoxLittleSchenck,
  title        = {Toric Varieties},
  author       = {Cox, D.~A. and Little, J.~B. and Schenck, H.~K.},
  series       = {Graduate Studies in Mathematics},
  volume       = {124},
  year         = {2011},
  publisher    = {American Mathematical Society},
  address      = {Providence, RI},
  isbn         = {978-0-8218-4819-7},
  doi          = {10.1090/gsm/124},
  pages        = {841},
}

@article{SturmfelsTelek,
author = {Sturmfels, B. and Telek, M.~L.},
 title={Copositive geometry of {F}eynman integrals}, 
journal = {Lett. Math. Phys.},
volume = {115},
number = {74},
year = {2025}
}

@book{Ziegler,
AUTHOR = {Ziegler, G\"unter M.},
     TITLE = {Lectures on polytopes},
    SERIES = {Graduate Texts in Mathematics},
    VOLUME = {152},
 PUBLISHER = {Springer-Verlag, New York},
      YEAR = {1995},
}

@book{TelenBook,
author = {Telen, S.},
title = {Applied toric geometry},
year = {2025},
note={Draft version available at \\ https://sites.google.com/view/simontelen/teaching}
}

@article{Borinsky,
      author = {Borinsky, Michael},
 title = {Tropical {Monte} {Carlo} quadrature for {Feynman} integrals},
 fjournal = {Annales de l'Institut Henri Poincar{\'e} D. Combinatorics, Physics and their Interactions},
 journal = {Ann. Inst. Henri Poincar{\'e} D, Comb. Phys. Interact.},
 volume = {10},
 number = {4},
 pages = {635--685},
 year = {2023},
}

@article{Bruns,
    AUTHOR = {Bruns, W.},
     TITLE = {Polytope volume in {N}ormaliz},
   JOURNAL = {S\~ao Paulo J. Math. Sci.},
  FJOURNAL = {S\~ao Paulo Journal of Mathematical Sciences},
    VOLUME = {17},
      YEAR = {2023},
    NUMBER = {1},
     PAGES = {36--54},
}

@article{Krivine,
      title={Anneaux préordennés}, 
      author={Krivine, J.~L.},
      year={1964},
      journal={J. Analyse Math.},
    volume={12},
    pages={307–326}
}

@article{HennRaman,
title = {Positivity properties of scattering amplitudes},
author = {Henn, J. and Raman, P.},
journal = {J. High Energ. Phys.},
year = {2025},
volume = {150}
}

@article{HennSmirnovSmirnov,
title = {Evaluating single-scale and/or non-planar diagrams by differential equations},
author = {Henn, J. and Smirnov, A.~V.  and Smirnov, V.~A.},
journal = {J. High Energ. Phys.},
year = {2014},
volume = {88}
}

@article{ConradiFeliuWiuf,
      title={Identifying parameter regions for multistationarity}, 
      author={Conradi, C. and Feliu, E. and Mincheva, M. and Wiuf, C.},
      year={2017},
      journal={PLOS Comput. Biol.},
    volume={13},
number = {10}
}

@book{Weinzierl,
title = {Feynman Integrals: A Comprehensive Treatment for Students and Researchers},
  journal = {UNITEXT for Physics},
  publisher = {Springer International Publishing},
  author = {Weinzierl,  Stefan},
  year = {2022}
}

@article{CastlePowersReznick2009,
  author  = {Castle, M. and Powers, V. and Reznick, B.},
  title   = {A Quantitative {P}ólya’s Theorem with Zeros},
  journal = {J. Symb. Comput.},
  volume  = {44},
  number  = {9},
  pages   = {1285--1290},
  year    = {2009}
}

@book{Scheiderer2024,
  title = {A Course in Real Algebraic Geometry: Positivity and Sums of Squares},
  series = {Graduate Texts in Mathematics},
  publisher = {Springer International Publishing},
  author = {Scheiderer,  C.},
  year = {2024}
}

@article{Bihan2002,
  title = {Viro Method for the Construction of Real Complete Intersections},
  volume = {169},
  ISSN = {0001-8708},
  url = {http://dx.doi.org/10.1006/aima.2001.2059},
  DOI = {10.1006/aima.2001.2059},
  number = {2},
  journal = {Adv. Math.},
  publisher = {Elsevier BV},
  author = {Bihan,  F.},
  year = {2002},
  month = aug,
  pages = {177–186}
}

@article{Burgdorf2012,
  title = {Pure states,  nonnegative polynomials and sums of squares},
  volume = {87},
  ISSN = {1420-8946},
  url = {http://dx.doi.org/10.4171/CMH/250},
  DOI = {10.4171/cmh/250},
  number = {1},
journal = {Comment. Math. Helv.},
  fjournal = {Commentarii Mathematici Helvetici},
  publisher = {European Mathematical Society - EMS - Publishing House GmbH},
  author = {Burgdorf,  S. and Scheiderer, C. and Schweighofer,  M.},
  year = {2012},
  month = jan,
  pages = {113–140}
}


\end{document}